\documentclass[11pt]{amsart}

\usepackage{amssymb}
\usepackage{graphicx}

\usepackage[colorlinks=true,urlcolor=blue]{hyperref}
\usepackage{geometry}

\def\textcolor#1{}

\newcommand{\N}{\mathbb{N}}

\newcommand{\C}{\mathbb{C}}

\newcommand{\disk}{\mathbb{D}}
\newcommand{\Cc}{\widehat{{\C}}}
\newcommand{\Circle}{{\mathbb S}^1}

\newcommand{\Deltap}{\Delta^+}
\newcommand{\Deltat}{\Delta^t}
\newcommand{\inter}{\mathring}

\DeclareMathOperator{\fib}{fib}
\DeclareMathOperator{\modu}{mod}

\renewcommand{\tilde}{\widetilde}
\renewcommand{\rho}{\varrho}
\renewcommand{\phi}{\varphi}

\renewcommand{\theta}{\vartheta}

\newcommand\ovl[1]{\overline{#1}}
\newcommand{\sm}{\setminus}

\renewcommand{\ge}{\geqslant}
\renewcommand{\le}{\leqslant}

\theoremstyle{theorem}
\newtheorem{theorem}{Theorem}[section]
\newtheorem{lemma}[theorem]{Lemma}
\newtheorem{proposition}[theorem]{Proposition}
\newtheorem{corollary}[theorem]{Corollary}

\newtheorem{claim}{Claim}

\newtheorem{MainTheorem}{Theorem}

\theoremstyle{definition}
\newtheorem{definition}[theorem]{Definition}

\theoremstyle{remark}
\newtheorem*{remark}{\textsc{Remark}}

\newcounter{reminder}

\newcommand{\hide}[1]{}

\numberwithin{equation}{section}

\title{Puzzles and the Fatou--Shishikura injection for rational Newton maps}

\author{Kostiantyn Drach}
\author{Russell Lodge}
\author{Dierk Schleicher}
\author{Maik Sowinski}

\address{Aix--Marseille Universit\'e, Institut de Math\'ematiques de Marseille, 163 Avenue de Luminy, 13009 Marseille, France}
\email{kostya.drach@gmail.com}

\address{Department of Mathematics and Computer Science, Indiana State University, Terre Haute, IN 47809, USA}
\email{russell.lodge@indstate.edu}

\address{Aix--Marseille Universit\'e, Institut de Math\'ematiques de Marseille, 163 Avenue de Luminy, 13009 Marseille, France}
\email{dierk.schleicher@gmx.de}

\address{Universit\"at Bielefeld, Universit\"atsstrasse 25, 33615 Bielefeld, Germany}
\email{maik.sowinski@gmx.de}

\thanks{This research was partially supported by the advanced grant 695\,621 ``HOLOGRAM'' of the European Research Council (ERC), which is gratefully acknowledged. We thank the members of our dynamics team, especially Wolf Jung, for many discussions and comments on the topic of the paper. We are grateful to the anonymous referee for helpful comments that led to improvements of the paper. We are also grateful to Cornell University for their hospitality during our visits where much of this work was developed. }



\begin{document}

\begin{abstract}
We establish a principle that we call the \emph{Fatou--Shishikura injection} for Newton maps of polynomials: there is a dynamically natural injection from the set of non-repelling periodic orbits of any Newton map to the set of its critical orbits. This injection obviously implies the classical Fatou--Shishikura inequality, but it is stronger in the sense that every non-repelling periodic orbit has \emph{its own} critical orbit. 

Moreover, for every Newton map we associate a forward invariant graph (a \emph{puzzle}) which provides a dynamically defined partition of the Riemann sphere into closed topological disks (\emph{puzzle pieces}). This puzzle construction is for rational Newton maps what Yoccoz puzzles are for polynomials: it provides the foundation for all kinds of rigidity results of Newton maps beyond our Fatou--Shishikura injection. Moreover, it gives necessary structure for a classification of the postcritically finite maps in the spirit of Thurston theory.
\end{abstract}

\maketitle

\section{Introduction}

Over the past several decades, substantial progress has been made in the understanding of the dynamics of iterated rational maps, with a particular focus on the dynamics of polynomials: the invariant basin of infinity, and the resulting dynamics in B\"ottcher coordinates, provide strong tools for global coordinates of the dynamics and subsequently for very deep studies of the dynamical fine structure. In comparison, the understanding of the dynamics of non-polynomial rational maps is lagging behind in a number of ways. We believe that Newton maps of polynomials are not only dynamically well-motivated as root finders (see for instance \cite{HSS,BAS,NewtonEfficient,NewtonRobinMarvin} for recent progress) but are also a most suitable class of rational maps for which significant progress in analogy to polynomials is possible. Therefore Newton maps are a large and important family of rational maps. 
This paper provides some results --- in particular one that we call the Fatou--Shishikura injection --- as well as foundations for subsequent results that extend our understanding of polynomial maps to Newton maps (and possibly further rational maps). These include the following:
\begin{itemize}
\item
a \emph{classification of all postcritically finite Newton maps}; this is given in \cite{LMS1,LMS2}. This establishes the first large family of rational maps of all degrees, beyond polynomials \cite{Poirier}, for which such a classification is available;
\item
the \emph{rational rigidity principle} for Newton maps: in the dynamical plane, any two orbits in the Julia set of a Newton map can be combinatorially distinguished, except when they are related to actual polynomial dynamics. Such results are often described in terms of ``local connectivity'' of the Julia set, while a stronger and more precise way to express this is ``triviality of fibers'' in the Julia set (compare \cite{FibersMandel}). Furthermore, in parameter space, any two combinatorially equivalent Newton maps are quasi-conformally conjugate provided they are either non-renormalizable, or both renormalizable ``in the same way''. These results are usually called ``rigidity'' and are described in \cite{RationalRigidity1}. 
\end{itemize}

Results like these might lead one to think that, contrary to frequent belief, rational dynamics may not be much more complicated than polynomial dynamics as soon as a good combinatorial structure is established. An earlier application of this philosophy could be found in the case of cubic Newton maps, a one-dimensional family of maps with only one ``free'' critical point. The latter fact helped to develop a deep understanding of this family and its parameter space: \cite{Tan}, building on \cite{He}, produced combinatorics leading to a fruitful study of local connectivity and rigidity in \cite{Roe, RWY} (see also \cite{AR}). Our results, in this and subsequent manuscripts, extend previous work by moving from degree three to arbitrary degrees, and hence by moving from complex one-dimensional to arbitrary finite-dimensional parameter spaces. This is in analogy to the development from iterated quadratic polynomials (a fairly well understood one-dimensional family of unicritical maps) to polynomials of arbitrary degrees (the multicritical case is much less understood). 

\subsection{The main results}

For a given polynomial $p \colon \C \to \C$, the \textit{Newton map of $p$} is a rational map $N_p \colon \Cc \to \Cc$ defined as $N_p(z):=z - p(z)/p'(z)$. Such maps naturally arise in Newton's iterative method for finding roots of $p$. We will assume that the degree of $N_p$ is at least $3$. It is well possible that a Newton map coming from an entire transcendental function is still a rational map; in Section~\ref{Sec:PCFandFuture} we argue that our results apply to these as well. 

Before stating our first theorem, let us recall the classical Fatou--Shishikura inequality \cite{MitsuFatouShishikura}: every rational map of degree $d$ on the Riemann sphere has at most $2d-2$  non-repelling periodic orbits. A somewhat stronger reformulation is that the rational map has no more non-repelling periodic orbits than it has critical orbits \cite{AdamFSI}. Our first theorem is an upgrade of the count between non-repelling and critical orbits to the dynamically more significant result that \emph{every non-repelling periodic orbit has its own critical orbit}. For this assignment of ``its own critical orbit'' we introduce the term \emph{Fatou--Shishikura injection}. 
The underlying idea that for polynomials every indifferent orbit has its own critical orbit is due to Kiwi \cite{Kiwi}.
In order to make our first result precise, we use the following definition.

\begin{definition}[Dynamically associated critical orbit]
\label{Def:DynamicallyAssociated}
We say that a critical orbit is \emph{dynamically associated} to a non-repelling periodic orbit if the $\omega$-limit set of the critical orbit contains the non-repelling orbit (or, in the case of a Siegel point, the entire boundary of all the periodic Siegel components in the cycle). The critical orbit is \emph{exclusively dynamically associated} to a non-repelling periodic orbit if it is dynamically associated to this orbit but to no other non-repelling orbit. 
\end{definition}

Recall that the \emph{$\omega$-limit set of the orbit of a point $z \in \Cc$} is the set 
\[
\omega(z) := \bigcap_{n \in \mathbb N} \ovl{\left\{N_p^{\circ k}(z) \colon k > n\right\}}
\;.
\]

\begin{MainTheorem}[The Fatou--Shishikura injection for Newton maps of polynomials]
\label{Thm:FSI}
\label{Thm:FatouShishikura}
For every Newton map of a polynomial there is an injection from the set of its non-repelling periodic orbits to the set of its critical orbits that assigns to each non-repelling orbit a critical orbit that is exclusively dynamically associated to it. 
\end{MainTheorem}

\begin{remark}
There are some claims that we do \emph{not} make in this theorem. 
It may well happen that some critical orbit is dynamically associated to more than one non-repelling orbit (not exclusively so). For example, some critical orbits might be dense in the Julia set, so they would be associated to all indifferent periodic orbits. However, such critical orbits will then not be accounted for in our injection. Moreover, we do not claim that the injection described in the theorem is unique: there may indeed be several critical orbits that are exclusively dynamically associated to one non-repelling orbit. 
\end{remark}

\begin{remark}
The Fatou--Shishikura inequality in its original formulation, as introduced above, compares the number of non-repelling orbits to the number of critical orbits. Since then, some sort of art has developed to include ever more dynamical features in this inequality, especially for polynomials. For instance, the number of repelling periodic orbits that are not landing points of periodic dynamic rays can be added to the number of non-repelling orbits, as well as the number of \emph{wandering triangles} (triples of rays that land together at a point that is not eventually periodic); see \cite{BCLOS} for details and further results. In fact, these repelling orbits without rays and the wandering triangles again have ``their own'' critical orbits.

One can extend this injection also in our case; for instance, the polynomial results from \cite{BCLOS} can be imported into our setting in a straightforward way. 
Another dynamical feature that can be accounted for in our injection is infinitely renormalizable dynamics: every instance of infinitely renormalizable dynamics ``consumes'' at least one infinite critical orbit that is not dynamically associated to non-repelling orbits.

Let us also mention that Kiwi's results have recently been extended to the setting of entire transcendental maps \cite{BF1}, and a transcendental version of the Fatou--Shishikura inequality was proven in \cite{BF2} based on this extension (see also \cite{BF3} for a further refinement of the inequality involving so-called rationally invisible repelling orbits).

\end{remark}

We mentioned earlier the developing ``philosophy'' that difficulties on the dynamics of rational maps can be resolved by results on the dynamics of polynomials together with good combinatorial control on rational maps. 
The proof of Theorem~\ref{Thm:FSI} provides a good example for this; here is an outline of the arguments involved.

\begin{enumerate}
\item
The Fatou--Shishikura injection holds true for polynomials.
\item
For every Newton map $N_p$ of a polynomial, every non-repelling periodic point is either a (super-) attracting fixed point (a root of $p$), or it is contained in a domain of renormalization (defined below).
\item
The process of renormalization preserves the Fatou--Shishikura injection.
\end{enumerate}

The key step in this chain of arguments is to establish (2), and will be done by using a properly defined \emph{(Newton) puzzle partition}. This is not an obvious task: unlike in the polynomial case, where the basin of a superattracting fixed point at $\infty$ is partitioned in a straightforward way by equipotentials and rays landing together (the classical \emph{Yoccoz puzzle} construction for polynomials), rational maps in general do not have such a \textit{global} combinatorial structure. Our second main result of this paper --- Theorem \ref{Thm:B} --- shows that for Newton maps this difficulty can be resolved. It roughly says that \textit{every Newton map of a polynomial gives rise to a well-defined puzzle partition of arbitrary depth}. More precisely it says the following (all terms will be defined in later sections).

\begin{MainTheorem}[Newton puzzles for Newton maps of polynomials]
\label{Thm:B}
Every Newton map of a polynomial has an iterate $g$ for which there exists a finite graph $\Gamma \subset \Cc$  that is $g$-invariant (except possibly in a Fatou neighborhood of the roots), and so that for every $n \geqslant 0$  the complementary components of $g^{-n}(\Gamma)$ that intersect the Julia set are Jordan disks that satisfy the Markov property under $g$. These disk components define a Newton puzzle partition of depth $n$.
\end{MainTheorem}

The possible exception to forward invariance of $\Gamma$ is to be understood as follows: every root has a compact and forward invariant neighborhood within its Fatou component (the immediate basin) in which $\Gamma$ may fail to be  invariant.

Our construction of puzzles in Theorem~\ref{Thm:B}, restricted to cubic Newton maps, is different from the one in \cite{Roe}. Theorem~\ref{Thm:B} provides the foundation for much of our subsequent work using puzzles theory on rigidity of Newton maps; see in particular~\cite{RationalRigidity1}.

The paper is organized as follows. We start by reviewing some known facts about Newton maps in Section~\ref{Sec:Review}. After that, in Section~\ref{Sec:Renorm} we establish Theorem \ref{Thm:B} (in fact, in a slightly more general form, see Theorem~\ref{Thm:PuzPr}). The Newton puzzle partition is then exploited to identify renormalization domains for points that have periodic itineraries with respect to this partition, and to extract corresponding polynomial-like maps. The results of Section~\ref{Sec:Renorm} will be then used in Section~\ref{Sec:ProofFSI} to establish Theorem~\ref{Thm:FSI}. In order to make the paper self-contained, in Section~\ref{Sec:ProofFSI} we describe a proof of the Fatou--Shishikura injection for polynomials; it is based on the Goldberg--Milnor fixed point portraits, similarly as in \cite{Kiwi}. 

Substantial parts of this work are based on the Bachelor thesis of the last named author \cite{MaikBachelor}.

\section{Background on Newton maps}
\label{Sec:Review}

Let $N_p$ be the Newton map of a polynomial $p$ and let $d$ be the degree of $N_p$. It is straightforward to check that the fixed points of $N_p$ in $\C$ are exactly the distinct roots of $p$. Every such fixed point is attracting with multiplier $(m-1)/m$, where $m \ge 1$ is the multiplicity of this point as a root of $p$. In particular, simple roots are superattracting fixed points of $N_p$. Every Newton map has one more fixed point at $\infty$; it is repelling with multiplier $\deg p/(\deg p - 1)$.

A polynomial $p$ and its Newton map $N_p$ have the same degree if and only if all roots of $p$ are simple. In general, the degree of the Newton map equals the number of distinct roots of $p$. Since the case $d=2$ is trivial, we will assume that $d \geqslant 3$ without explicit mention from now on.

The rational maps that arise as Newton maps can be described explicitly as follows (see \cite[Proposition 2.1.2]{He}, as well as \cite[Proposition~2.8]{RS07} for a proof):
\begin{proposition}[Head's theorem]
\label{Prop_Head}
A rational map $f$ of degree $d \geqslant 3$ is a Newton map if and only if $\infty$ is a repelling fixed point of $f$ and for each fixed point $\xi\in\C$, there exists an integer $m\geqslant 1$ such that $f'(\xi)=(m-1)/m$.
\qed
\end{proposition}

For a map $f$ and $k \ge 1$ we write $f^{\circ k}$ for the $k$-th iterate of $f$. Conversely, for a set $A$ we denote by $f^{-k}(A)$ the full $k$-fold preimage of $A$, i.e.\ the full preimage of $A$ under $f^{\circ k}$.

Our first definition concerns the Fatou components of a Newton map $N_p$ that contain a root; these play a fundamental role in the study of Newton maps.

\begin{definition}[Immediate basin]
\label{Def_ImmediateBasin}
Let $N_p$ be a Newton map and $\xi\in\C$ a fixed point of $N_p$. Let $B_{\xi}=\{z\in\C\,:\, \lim_{n\to\infty}N_p^{\circ n}(z)=\xi\}$ be the \emph{basin (of attraction)} of $\xi$.
The connected component of $B_\xi$ containing $\xi$ is called the \emph{immediate basin} of $\xi$ and denoted $U_\xi$.
\end{definition}

Clearly, $B_\xi$ is open. By a theorem of Przytycki \cite{Pr}, $U_{\xi}$ is simply connected and $\infty\in\partial U_\xi$ is an accessible boundary point; in fact, a result of Shishikura \cite{Sh} implies that every component of the Fatou set of $N_p$ is simply connected.

\begin{definition}[Access to $\infty$]
\label{Def_Access}
Let $U_{\xi}$ be the immediate basin of the attracting fixed point $\xi \in \C$. For every injective curve $\Gamma:[0,1]\to U_{\xi}\cup\{\infty\}$ with $\Gamma(0) =\xi$ and $\Gamma(1)=\infty$, its homotopy class within $U_{\xi}\cup\{\infty\}$, fixing endpoints, defines an \emph{access to $\infty$} for $U_{\xi}$.
\end{definition}

In topologically simple cases, a simpler definition suffices.

\begin{definition}[Accesses to vertices of graphs]
\label{Def:SimpleAccess}
For a finite graph $\Gamma$ embedded in the sphere, an \emph{access to a vertex} $x\in \Gamma$ is given in terms of a (sufficiently small) disk $D$ around $x$: an access is then represented by a component of $D\sm\Gamma$ that contains $x$ on the boundary.  
\end{definition}

Most of the time we will use the simple definition of an access to a vertex of a finite graph. However, topological accesses to infinity within the immediate basins provides the important first-level combinatorial data due to the following proposition.

\begin{proposition}[Accesses to infinity; {\cite[Prop.~6]{HSS}}]
\label{Prop:AccessesHSS}
Let $N_p$ be a Newton map of degree $d\geqslant 3$ and $U_\xi$ an immediate basin for $N_p$. Then there exists $k_\xi \in \{1,\dots,d-1\}$ such that $U_\xi$ contains $k_\xi$ critical points of $N_p$ (counting multiplicities), $N_p|_{U_\xi}$ is a branched covering map of degree $k_\xi+1$, and $U_\xi$ has exactly $k_\xi$ accesses to $\infty$.
\qed
\end{proposition}

A point $z \in \C$ is called a \emph{pole} if $N_p(z)=\infty$, a \emph{prepole} if $N_p^{\circ k}(z)=\infty$ for some $k> 1$, and a \emph{pre-fixed} point if $N_p^{\circ k}(z)$ is a finite fixed point for some $k > 1$.

The first and fundamental step to construct our Newton puzzles is to construct what we call the \emph{channel diagram} (see below); this is a finite forward invariant graph that connects all fixed points of the Newton map. Strictly speaking, it only exists in the (not very) special case of  \emph{attracting-critically-finite} maps, and it is most convenient to work in this case. We will explain in Section~\ref{SSec:BeyondACF} how to adjust the definition of Newton puzzles in the general case.

\begin{definition}[Attracting-critically-finite]
\label{Def:acf}
We say that a Newton map $N_p$ of degree $d$ is \emph{attracting-critically-finite} if all critical points in the basins of the roots have finite orbits, or equivalently, all attracting fixed points are superattracting and all critical orbits in their basins eventually terminate at the fixed points. 
\end{definition}

The following observation is well known and its proof is standard.
\begin{lemma}[Only one critical point]
\label{Lem_OnlyCritical}
Let $N_p$ be a Newton map that is attracting-critically-finite and let $\xi\in\C$ be a fixed point of $N_p$ with immediate basin $U_{\xi}$. Then $\xi$ is the only critical point in $U_{\xi}$. 
\end{lemma}
\begin{proof}[Sketch of proof]
Let $\phi\colon U_\xi\to\disk$ be a Riemann map with $\phi(\xi)=0$. Then $f:=\phi\circ N_p\circ\phi^{-1}$ is a self-map of the standard unit disk $\disk$ with fixed point $0$ and degree at least $2$, and every $z\in\disk\sm\{0\}$ has infinite orbit converging to $0$. Therefore every $z\in U_\xi\sm\{\xi\}$ has infinite orbit converging to $\xi$.
\end{proof}

The \textit{postcritical set} of $N_p$ is the closure of the union of all forward iterates of all critical points of $N_p$. A map is called \textit{postcritically finite} if its postcritical set is finite. Clearly, every postcritically finite Newton map is attracting-critically-finite, but the latter class of maps is much larger. In fact, using Head's theorem (Proposition~\ref{Prop_Head}) and a routine quasiconformal surgery in the basins of roots, one can prove the following (see \cite[Section 3]{DMRS}):

\begin{proposition}[Making Newton map attracting-critically-finite]
\label{Prop:ACF}
For every polynomial $p$ there exists a polynomial $\tilde p$ and a quasiconformal homeomorphism $\tau \colon \Cc \to \Cc$ with $\tau(\infty)=\infty$ such that:
\begin{enumerate}
\item
$\deg N_p = \deg N_{\tilde p} = \deg \tilde p \leqslant \deg p$; 
\item
the Newton map $N_{\tilde p}$ is attracting-critically-finite;
\item
$\tau$ conjugates $N_p$ and $N_{\tilde p}$ in some neighborhood of the Julia sets of $N_p$ and $N_{\tilde p}$ union all Fatou components (if any) that do not belong to the basins of roots.\qed
\end{enumerate}
\end{proposition}

\begin{remark}
The homeomorphism $\tau$ can be constructed so that it has vanishing dilatation on the Julia set (which is relevant only if the latter has positive measure), as well as on Fatou components away from the root basins. 
\end{remark}

Using Proposition~\ref{Prop:ACF} we can focus our attention on the attracting-critically-finite Newton maps. For such maps all finite fixed points are superattracting, and they are the only critical points in their respective immediate basins (Lemma \ref{Lem_OnlyCritical}). Denote these points by $\xi_1, \xi_2, \ldots, \xi_d$. For such maps we can define a \emph{Newton graph} (see \cite[Section 2]{DMRS}) the construction of which we will now recall.

Let $U_i$ be the immediate basin of $\xi_i$. Then $U_i$ has a global B\"ottcher coordinate $\phi_i\colon(\disk,0) \to (U_i, \xi_i)$  with the property that $N_p(\phi_i(z))=\phi_i(z^{k_i})$ for each $z \in \disk$; here $k_i-1 \geqslant 1$ is the multiplicity of $\xi_i$ as a critical point of $N_p$. The map $z \mapsto z^{k_i}$ fixes $k_i-1$ rays in $\disk$. Under $\phi_i$, these are mapped to $k_i-1$ pairwise disjoint (except for endpoints) simple curves $\Gamma^1_i,\Gamma^2_i,\ldots,\Gamma^{k_i-1}_{i}\subset U_i$ that connect $\xi_i$ to $\infty$, are pairwise non-homotopic in $U_i$ (with homotopies fixing the endpoints) and are invariant under $N_p$. They represent all accesses to $\infty$ of $U_i$ (see Proposition \ref{Prop:AccessesHSS}).

\begin{definition}[Channel diagram of a Newton map]\label{Def_ConcreteChannelDiagram}
The \emph{channel diagram} $\Delta$ associated to an attracting-critically-finite Newton map $N_p$ is the finite connected graph with  vertex set $\{\infty, \xi_1, \xi_2,...,\xi_d\}$ and edge set
\[\displaystyle \bigcup_{i=1}^d \bigcup_{j=1}^{k_{i}-1}
\Big\{\overline{ \Gamma^{j}_{i} }\Big\}.\]
\end{definition}

Clearly $N_p(\Delta)= \Delta$. The channel diagram records the mutual locations of the immediate basins of $N_p$ and provides a first-level combinatorial information about the dynamics of the Newton map. It turns out that the channel diagram carries more information about location of the roots.

\begin{proposition}[Complement of immediate basin; {\protect\cite[Corollary~5.2]{RS07}}]
\label{Prop:FixedPoles}
For every immediate basin $U_\xi$ of a Newton map, every component of $\C\setminus U_\xi$ contains at least one fixed point.\qed
\end{proposition}

This proposition implies that the channel diagram has the property that for any pair of fixed rays within the same immediate basin, both complementary components contain at least one further vertex of $\Delta$, i.e.\ one root of $p$. A more precise count on the number of such vertices can be found in \cite[Theorem 2.2]{DMRS}.

\begin{definition}[Level $n$ Newton graph]\label{Def_ConcreteNewtonGraph}
For any $n\ge 0$, denote by $\Delta_n$ the connected component of $N^{-n}_p(\Delta)$ that contains $\Delta$ (with $\Delta_0 := \Delta$). The graph $\Delta_n$ is called the \emph{Newton graph} of $N_p$ at level $n$. 
\end{definition}

By construction, the Newton graph is forward invariant, that is $N_p (\Delta_{n+1}) = \Delta_n \subset \Delta_{n+1}$ for every $n \geqslant 0$. Every edge of $\Delta_n$ is an internal ray of a component of some basin $B_{\xi_i}$, while every vertex is either $\xi_i$, or $\infty$, or an iterated preimage of these. Observe that vertices in $\Delta_n$ are alternating points in the Fatou and the Julia set of $N_p$.

We next state one of the key results in \cite{DMRS} (\cite[Theorem 3.4]{DMRS}; this result is the core of \cite[Theorem A]{DMRS}, which explains the structure of the Fatou set of general Newton maps). 

\begin{theorem}[Poles connect to $\infty$]
\label{Thm:PolesInGraph}
There is an $N\in\N$ such that $\Delta_n$ contains all poles of $N_p$ for all $n\ge N$.\qed
\end{theorem}

In the sequel, $N$ will denote the least such integer. 

As an immediate corollary we see that each prepole is contained in the Newton graph of sufficiently high level (see \cite[Corollary 3.5]{DMRS}).

\begin{corollary}[Prepoles in Newton graph]
\label{Cor:PrepolesInNewtonGraph}
Let $m \geqslant 0$ be an integer, and let $N$ be as in Theorem \ref{Thm:PolesInGraph}. Then every point in $N_p^{-(m+1)}(\infty)$ is a vertex of $\Delta_{m+N}$. \qed
\end{corollary}

\section{Renormalization of Newton maps}
\label{Sec:Renorm}


In this section we develop a combinatorial tool called \textit{Newton puzzle partition} that leads to the proof of Theorem~\ref{Thm:B}. Using this tool we then identify renormalization domains for Newton maps of polynomials, which will be a key step towards the proof of the Fatou--Shishikura injection (Theorem~\ref{Thm:FSI}).

\subsection{Separating circles in Newton graphs}

This subsection contains the key technical result of the section, Proposition \ref{Prop:Circles}. It shows that for all sufficiently high levels $n$ all critical points of $N_p$ are separated from $\infty$ by a subset of $\Delta_n$ given by an \emph{almost} disjoint union of topological circles (where \emph{almost} means that all intersections of these circles coincide with the set of finite fixed points of $N_p$; see Figures~\ref{Fig:Real} and~\ref{Fig:Circles}). This result will prove its importance in Subsection~\ref{SSec:Puz}, where it will be used to resolve boundary pinching problems for would-be puzzle pieces. It also implies that the Julia set is locally connected, and even has trivial fibers, at $\infty$ and all poles and prepoles.

\begin{proposition}[Newton graphs circle-separated]
\label{Prop:Circles}
For every attracting-critically-finite Newton map $N_p$ there exists an index $K \geqslant 1$ such that for every component $V$ of\/ $\Cc \sm \Delta$ there exists a topological circle $X_V \subset \Delta_K \cap \ovl V\cap \C$ that passes through all finite fixed points in $\partial V$, separates $\infty$ from all critical values of $N_p$ in $V$, and does not contain a point on a critical orbit, except the roots.
\end{proposition}

Such a circle $X_V$ connects the roots on $\partial V$ in the circular order induced by the channels to $\infty$ (this circular order is  well defined even when some roots on $\partial V$ have several channels: at most two of these channels are in $\partial V$, and if there are two then they are adjacent in the circular order as far as $V$ is concerned). The connection of adjacent channels is trivial when their immediate basins have a common pole or prepole. If not, then the immediate basins are connected by finitely many edges in $\Delta_K$ called ``bridges'', close enough to $\infty$ so that the resulting circle surrounds all critical values in $V$. In other words, for any two roots with adjacent channels, they have a fixed ray each within  the channels that together bound an access of $V$ to $\infty$, and the bridge is a perturbation within $V\subset\C$ of these two fixed rays, small enough so that the conclusion concerning the critical values is satisfied. 
The union of these bridges for all pairs of adjacent channels with their circular order yields the circle $X_V$. Pictures of such circles are shown in Figure~\ref{Fig:Real} (within the dynamics of an actual Newton map, in the  special case that every root  has a single channel, i.e.\ $\Delta$ does not disconnect $\C$) and in Figure~\ref{Fig:Circles} (sketch of the graph when $\Delta$ does disconnect $\C$). 
 
The proof of Proposition~\ref{Prop:Circles}, and especially its key Lemma~\ref{Lem:SweepOut}, parallels that of \cite[Theorem 3.4]{DMRS} in many ways, but we cannot simply quote these results without losing  properties that we need to keep track of. Hence, in order to make our presentation self-contained and more readable, we have to reproduce some arguments from the proof of \cite[Theorem 3.4]{DMRS}.

\begin{remark}
It is not hard to state the result of Proposition~\ref{Prop:Circles} for general Newton maps (in terms of touching Fatou components), using Proposition~\ref{Prop:ACF}.
\end{remark}

\begin{figure}[htbp]
\begin{center}
\includegraphics[scale=0.75, trim=20 20 20 20]{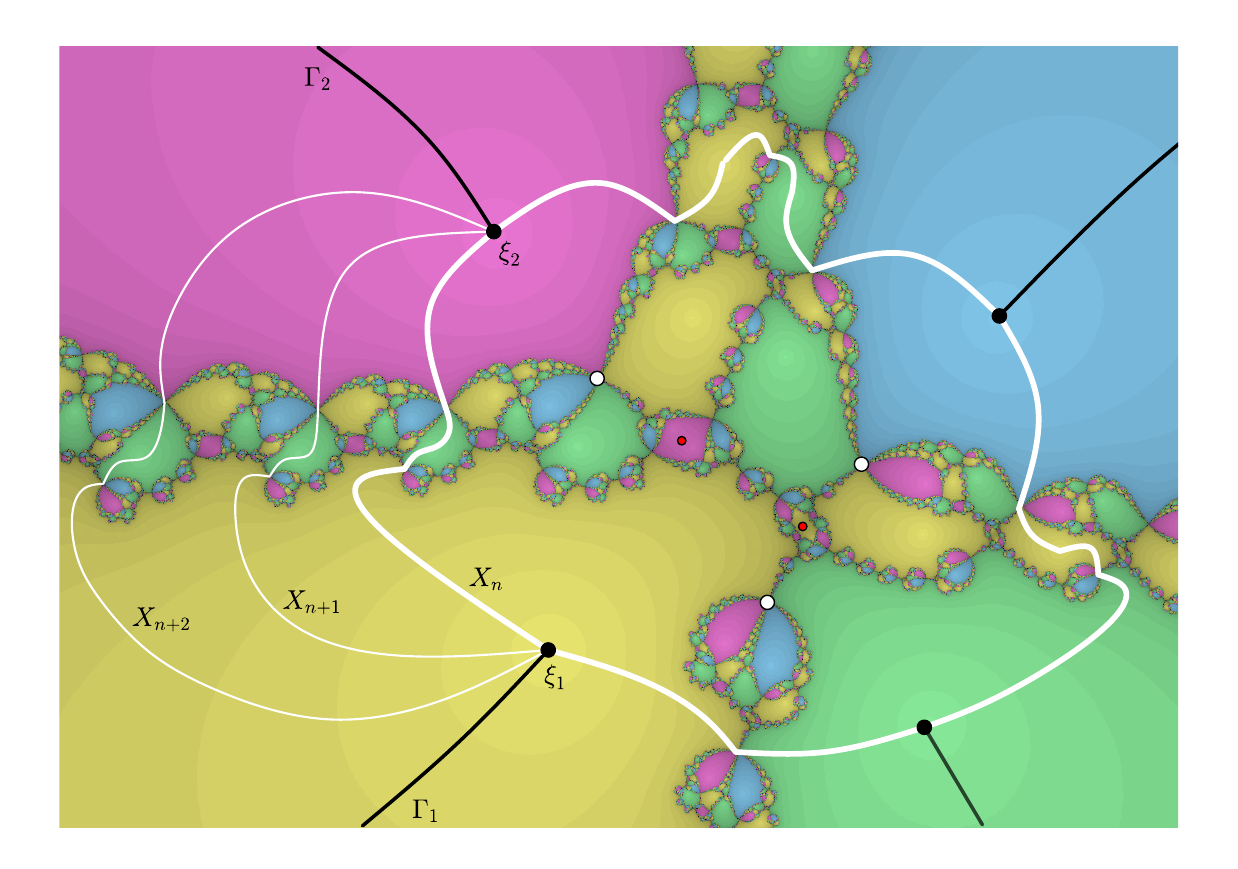}
\caption{Example of the dynamical plane of a degree 4 Newton map. The four roots are marked by black dots, the three poles by white dots, and the two  ``free'' critical points that are not roots by red dots. The channel diagram $\Delta$ is the union of the black lines. In this example, there is only one component $V$ of $\Cc \sm \Delta$. The topological circle $X_V \subset \Delta_n \cap \ovl V$ is drawn as a thick white line (for $n=2$). In thin white, $X_{n+1}$ and $X_{n+2}$ are two successive lifts of the bridge $X_n \subset X_V$ connecting $\xi_1$ and $\xi_2$.}
\label{Fig:Real}
\end{center}
\end{figure}

Proposition~\ref{Prop:Circles} immediately implies the following result for attracting-critically-finite Newton maps (this result was shown in \cite[Lemma 4.9]{DMRS} in the special case of \emph{postcritically fixed} Newton maps).

\begin{corollary}[Connectivity in $\C$]\label{Cor:ConnectedInC}
If $N_p$ is an attracting-critically-finite Newton map, then for all sufficiently large $n$, the graph $\ovl{\Delta_n\sm\Delta}$ is connected.\qed
\end{corollary}

We start our preparation for the proof of Proposition~\ref{Prop:Circles} by introducing some notation.
Let $\Gamma \subset \Cc$ be a connected graph and let $x$ be a vertex of $\Gamma$. Denote by $\mathcal A_x(\Gamma)$ the set of all accesses to $x$ for $\Gamma$ (see Definition~\ref{Def:SimpleAccess}). Likewise, if $G$ is a component of $\Cc \sm \Gamma$ with $x \in \partial G$, we define $\mathcal A_x(G)$ to be the set of all accesses to $x$ in $G$.

Since $N_p$ is a local homeomorphism in some neighborhood of $\infty$ and $\Delta$ is invariant under $N_p$, we have $\mathcal A_\infty(\Delta) = \mathcal A_\infty (\Delta_n)$ for every $n \geqslant 0$; we shorten the notation by setting $\mathcal A_\infty := \mathcal A_\infty(\Delta)$. Every access in $ \mathcal A_\infty$ corresponds to a unique pair of adjacent fixed rays $\Gamma^k_i \neq \Gamma^l_j$ (here \textit{adjacent} means with respect to the cyclic order of edges at $\infty$). If $\xi_i$ and $\xi_j$ are the endpoints of these rays other than $\infty$, then $\xi_i \neq \xi_j$ by Proposition~\ref{Prop:FixedPoles}. Clearly at most two distinct accesses in $\mathcal A_\infty$ can be bounded by rays starting from the same pair of fixed points.

\subsubsection{Proving Proposition~\ref{Prop:Circles} with supporting lemmas}
We always assume  that $n \geqslant N$. For a given access $a \in \mathcal A_\infty$, let $V_n$ be the unique unbounded component of $\Cc \sm \Delta_n$ with $a \in \mathcal A_\infty(V_n)$. Define $V_{n+1}$ to be the component of $N_p^{-1}(V_n)$ that has the same access $a$. Such a component necessarily exists since the inverse branch of $N_p$ fixing $\infty$ preserves $a$. Moreover, since for all $n\ge N$ the Newton graph $\Delta_n$ contains all poles (Theorem~\ref{Thm:PolesInGraph}), the component $V_{n+1}$ is a topological disk. And since the Newton graph is forward invariant, we also have $V_{n+1} \subset V_n$.   

Inductively, we construct a sequence of nested topological disks $V_N \supset V_{N+1} \supset \ldots$ with $a \in \mathcal A_\infty(V_n)$ and $N_p \colon V_{n+1} \to V_n$ proper for every $n \geqslant N$. 

\begin{lemma}[Basic properties of $V_n$ for $n\ge n_a$]
\label{Lem:BasicProperties_V_n}
There is an index $n_a \ge N$ such that for $n\ge n_a$ all $V_n$ have the same accesses to $\infty$ and the same accesses to the same poles, and their boundaries contain the same fixed rays (within immediate basins).
\end{lemma}
\begin{proof}
Since $V_n$ are nested and $\partial V_n$ consists of edges in $\Delta_n$ and their endpoints, the statement about the fixed rays within immediate basins is clear. The same is true for the accesses to $\infty$: $\left(\mathcal A_\infty(V_n)\right)_{n=N}^\infty$ is a non-increasing (under inclusion) sequence of finite sets (bounded below by $\{a\}$), and accesses to $\infty$ are bounded by fixed rays in immediate basins. The claim about accesses to poles follows.
\end{proof}

The main ingredient in the proof of Proposition~\ref{Prop:Circles} is the following lemma, which we now state using the notation introduced above.

\begin{lemma}[Bridges accumulate only at two fixed rays]
\label{Lem:SweepOut}
There exists $k_a > n_a$ such that for each $n \ge k_a$ there exists a pair of roots $\xi_1, \xi_2$ together with their adjacent fixed rays $\Gamma_1, \Gamma_2$, all in $\partial V_n$, such that 
\begin{enumerate}
\item
\label{It:ML1}
there is a unique Jordan arc $X_n \subset \partial V_n \sm \Delta$, called a \emph{bridge}, that joins $\xi_1$ to $\xi_2$;
\item
\label{It:ML2}
$N_p(X_{n+1}) = X_n$ and the restriction $N_p \colon X_{n+1} \to X_n$ is a homeomorphism; 
\item
\label{It:ML4}
and if $W_n$ is the component of\, $\Cc \sm (\ovl X_n \cup \ovl \Gamma_1 \cup \ovl \Gamma_2)$ containing $V_n$, then 
\begin{equation}
\label{Eq:Intersect}
\bigcap_{n \ge k_a} \ovl W_n = \ovl \Gamma_1 \cup \ovl \Gamma_2.
\end{equation}
\end{enumerate}
\end{lemma}

\begin{remark}
Bridges connect the boundaries of the immediate basins of roots through chains of edges of $\Delta_n\cap\C$ (i.e.\ within components of the basins); see Figure~\ref{Fig:Real} for an illustration. Their existence is an ever-important feature of the dynamics of Newton maps.  Since $\ovl X_n\cup\ovl\Gamma_1\cup\ovl \Gamma_2\subset\partial V_n$, we have $V_n\subset W_n$, but this may be  a proper inclusion: there may be parts of $\partial V_n$ that ``stick in'' to $W_n$ (see Figure~\ref{Fig:Bridges}).
\end{remark}

We postpone the proof of Lemma~\ref{Lem:SweepOut} to the next subsection, and first derive the following corollary to the lemma and use both the corollary and the lemma to prove Proposition~\ref{Prop:Circles}.

\begin{corollary}[Unique access and distinct roots]
\label{Cor:UniqueAccess}
In the notation of Lemma~\ref{Lem:SweepOut}, for all $n\ge n_a$, the domain $V_n$ has a single access to $\infty$; this access is bounded by the two fixed rays $\Gamma_1, \Gamma_2 \subset \partial V_n$ starting at $\xi_1$ resp.\ $\xi_2$, and $\xi_1 \neq \xi_2$.
\end{corollary}
 
\begin{proof}
If there were more than one access to $\infty$ in $V_n$, then, since $V_n \subset W_n$, the set $W_n$ would contain a root distinct from $\xi_1$ and $\xi_2$ for all $n \ge k_a$. This is impossible because of Lemma~\ref{Lem:SweepOut} \eqref{It:ML4}. Hence $\mathcal A_\infty(V_n) = \{a\}$ for all $n \ge k_a$. But since for $n \ge n_a$ the number of accesses to $\infty$ in $V_n$ is constant (Lemma~\ref{Lem:BasicProperties_V_n}), we conclude that $\mathcal A_\infty(V_n) = \{a\}$ must hold for $n \ge n_a$. 

A similar reasoning applies to conclude $\xi_1 \neq \xi_2$: if that was not the case, then each of the unbounded components of $\Cc \sm (\ovl \Gamma_1 \cup \ovl \Gamma_2)$ would contain at least one root distinct from $\xi_1 = \xi_2$ (Proposition~\ref{Prop:FixedPoles}). This would imply $|\mathcal A_{\infty}(V_n)| > 2$ for all $n \ge k_a$, a contradiction. Again, the conclusion $\xi_1 \neq \xi_2$ must hold for all $n \ge n_a$.
\end{proof}

\begin{proof}[Proof of Proposition~\ref{Prop:Circles}]
Since by Corollary~\ref{Cor:UniqueAccess} the domains $V_n$ depend on the chosen access $a$, we adjust the notation for the corresponding bridges to $X_{a,n}$; these bridges are given by Lemma~\ref{Lem:SweepOut} \eqref{It:ML1}. Since $\xi_1 \neq \xi_2$ (Corollary~\ref{Cor:UniqueAccess}), the closure of each bridge is a Jordan arc.

For any two accesses $a\neq a' \in \mathcal A_\infty$, we may require $ X_{a,n} \cap X_{a', n} = \emptyset $ for all $n \ge k_a$, possibly by increasing $k_a$: this is because, by \eqref{Eq:Intersect}, any particular $X_{a, n}$ can intersect only finitely many elements of the sequence $(X_{a', n})$. Since each bridge is the preimage of the previous one under $N_p$ (Lemma~\ref{Lem:SweepOut} \eqref{It:ML2}), further increasing $k_a$ if necessary, we can make sure that no $X_{a,n}$ contains a point on the orbit of a critical point. By increasing $k_a$ even more, we can also guarantee that $W_n$ does not contain a critical value of $N_p$ for all $n \ge k_a$; this can be done because of Lemma~\ref{Lem:SweepOut} \eqref{It:ML4}. (Note that there might be critical values of $N_p$ in $W_n \sm V_n$; this is why the final increase of $k_a$ might be necessary.)

Finally, Proposition~\ref{Prop:Circles} follows by setting $K := \max_{a \in \mathcal A_\infty} k_a$. With this choice,
\[   
X_V := \bigcup_{a \in \mathcal A_{\infty} (V)} \ovl{X_{a,K}}
\] 
is the required topological circle for the component $V$ under consideration. This concludes the proof. 
\end{proof}

\subsection{Proving Lemma~\ref{Lem:SweepOut}}

Let us start with an outline of the proof together with its supporting lemmas. Our goal is for every sufficiently large $n$ to identify a pair of roots $\xi_1$ and $\xi_2$ in $\partial V_n$ together with a unique Jordan curve $X_n$ (a bridge) connecting these roots in $\partial V_n \cap \C$, and to show that the bridges are mapped homeomorphically to bridges and that the ``sweeping'' property \eqref{Eq:Intersect} holds true. In Subsection~\ref{SSec:BP}, we start by parameterizing the boundary of each $V_n$ by $\Circle$ respecting the dynamics of the Newton map. This parametrization will allow us to find a finite union $J \subset \Circle$ of intervals and a pair of roots $\xi_1$, $\xi_2$ together with their adjacent fixed rays $\Gamma_1, \Gamma_2$, all in $\partial V_n$, such that $J$ parameterizes the bridge $X_n \subset \partial V_n \cap \C$ which together with $\Gamma_1$ and $\Gamma_2$ forms a closed loop. This loop starts from $\infty$ to $\xi_1$ along the fixed internal ray $\Gamma_1$ in $U_{\xi_1}$, then traverses $X_n$, first along a non-fixed internal ray from $\xi_1$ to a point on $\partial U_{\xi_1}$, then along some other edges in $\Delta_n \cap \C$, and finally along another non-fixed internal ray from a point on $\partial U_{\xi_2}$ to $\xi_2$; the loop closes up along the fixed internal ray $\Gamma_2$ from $\xi_2$ to $\infty$. 
In this loop, the bridge will be parameterized using the degree $1$ interval $I_1 \subset \Circle$ from Subsection~\ref{SSec:Deg1}. It is straightforward to define a bridge in the case $\xi_1 \neq \xi_2$ (done in Subsection~\ref{SSec:Br1}), and it is more involved when $\xi_1 = \xi_2$ (done in Subsection~\ref{SSec:Br2}). In the latter case, additional analysis is required to define the bridge uniquely. This analysis is done in Subsections~\ref{SSec:Mapping} and~\ref{SSec:Inclusion}. Finally, in Subsection~\ref{SSec:Props}, we once again use the results of the latter two subsections to establish conclusion \eqref{It:ML2} and property~\eqref{Eq:Intersect} of the lemma. This will conclude the proof.
  
\subsubsection{Boundary parametrization.} 
\label{SSec:BP} 
For every $n \ge n_a \ge N$, the boundary $\partial V_n$ is the image of $\Circle$ under a piecewise analytic surjection $\gamma_n \colon \Circle \to \partial V_n$ (called a \textit{boundary parametrization} of $\partial V_n$) such that for every edge $e$ of $\partial V_n$ the preimage $\gamma_n^{-1}(e)$ is one or two disjoint intervals in $\Circle$ that are mapped diffeomorphically onto $e$ (the case of two intervals is realized if a particular boundary arc belongs to the boundary of $V_n$ on both sides) and such that the order in which $\gamma_n$ visits the edges of $\partial V_n$ is given by the traversal of $\partial V_n$ along the ideal boundary of $V_n$. It is no loss of generality to assume that $\gamma_n$ traverses $\partial V_n$ only once and in forward direction, in the sense that the winding number of $\gamma_n$ around an inner point of $V_n$ is equal to $1$. If we also have a boundary parametrization $\gamma_{n+1} \colon \Circle \to \partial V_{n+1}$ for $V_{n+1}$, there is  a unique orientation-preserving circle endomorphism $\tau_n \colon \Circle \to \Circle$ such that we have a commutative diagram
\begin{equation}
\label{Eq:FuncRel}
\begin{picture}(120,55)(0,0)
\put(0,40){$\Circle\quad  \rule{0pt}{0pt}\raisebox{4pt}{\vector(1,0){40}} \quad \rule{0pt}{0pt} \partial V_{n+1}$}
\put(29,49){$\gamma_{n+1}$}
\put(82,35){\vector(0,-1){25}}
\put(86,21){$N_p$}
\put(0,0){$\Circle \quad \raisebox{4pt}{\vector(1,0){40}} \quad  \partial V_n$}
\put(36,9){$\gamma_{n}$}
\put(5,35){\vector(0,-1){25}}
\put(9,21){$\tau_n$}
\end{picture}
\end{equation}
(this is easiest wherever $\gamma_n$ has a unique preimage of $N_p(\gamma_{n+1}(t))$, and such points are except at edges that are traversed in both directions). 
Since both $\gamma_n$ and $\gamma_{n+1}$ have winding numbers equal to $1$, it follows that $\deg \tau_n = \deg (N_p \colon V_{n+1} \to V_n)$. We call $\tau_n$ the \emph{map that connects the boundary parametrizations $\gamma_n$ and $\gamma_{n+1}$}. 
We will need the freedom to precompose $\tau_n$ with any orientation preserving homeomorphism; the diagram above clearly continues to commute by precomposing $\gamma_{n+1}$ with the same homeomorphism.

All $V_n$ have the same accesses to $\infty$ (Lemma~\ref{Lem:BasicProperties_V_n}); denote their number by $m$. Fix some boundary parametrization $\gamma_n \colon \Circle \to \partial V_n$ of $V_n$. This defines $3m$ points 
\[
t_1, s_1, t_1', t_2, s_2, t_2', \ldots, t_m, s_m, t_m' \in \Circle
\] 
such that 
\begin{itemize}
\item
$\{s_1, \ldots, s_m\} = \gamma_n^{-1}(\{\infty\})$;
\item
$\gamma_n(t_i)$ and $\gamma_n(t_i')$ are roots of $p$ for all $i$;
\item
$\gamma_n([t_i,s_i])$ and $\gamma_n([s_i,t'_i])$ are the two fixed rays connecting two roots to $\infty$ (see Figure~\ref{Fig:Possibilities}). The union of all these fixed rays contains the boundaries of all accesses in $A$, and, conversely, the boundary of every access in $A$ intersects a pair of such rays.
\end{itemize}
Indeed, the $s_i$ are uniquely defined as preimages of $\infty$, and since $\infty$ can be reached in $\partial V_n$ only  along fixed rays within immediate basins, this defines the points $t_1,\dots,t_m$ and $t'_1,\dots,t'_m\in\Circle$. 
These points are defined uniquely up to cyclic order once the boundary parametrization is chosen. It is clear that all $3m$ points are distinct, except perhaps $t'_i=t_{i+1}$. But this is impossible because for $n \ge 1$ there are no adjacent edges in $\Delta_n$ connecting the same root to $\infty$ (see Proposition~\ref{Prop:FixedPoles}), so all these $3m$ points are indeed distinct.
Therefore, the set $I_i := [t_i', t_{i+1}]$ is a non-degenerate interval for each $i \in \{1, \ldots, m\}$. 

For each $i$, the restriction $\gamma_n|_{[t_i, t_i']}$ is an embedding. 
However, it is possible that $\gamma_n(t_i') = \gamma_n(t_{i+1})$ or even $\gamma_n([s_i, t_i']) = \gamma_n([t_{i+1}, s_{i+1}])$ or $\gamma_n([t_i, s_i]) = \gamma_n([s_{i-1}, t_{i-1}'])$ (the same edge to $\infty$ may be parametrized twice, in reverse  orientation; see the example in Figure~\ref{Fig:Possibilities}).

\begin{figure}[htbp]
\begin{center}
\includegraphics[width=.75\textwidth, trim=32 20 35 28, clip]{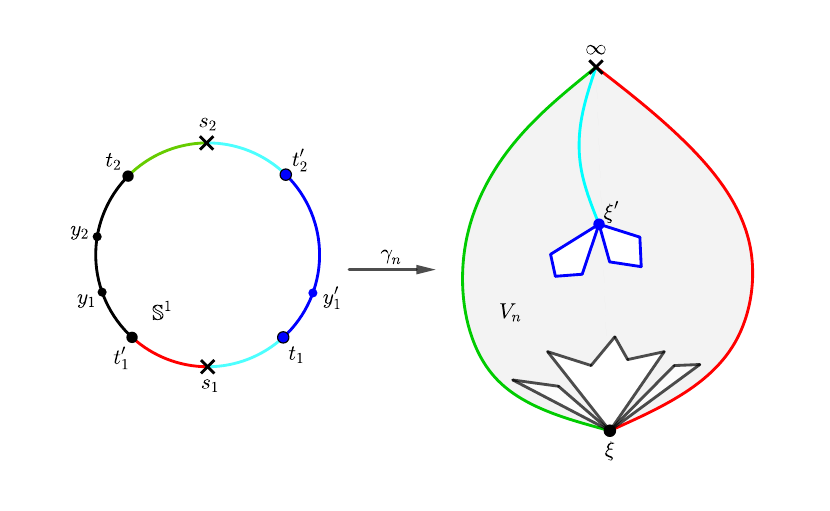}\hfill
\caption{Sketch of a situation, for $m=2$, when the boundary parametrization $\gamma_n \colon \Circle \to \partial V_n$ is such that $\gamma_n(t_1') = \gamma_n(t_2) = \xi$ and $\gamma_n\left([t_1,s_1]\right) = \gamma_n\left(\left[s_2, t_2'\right]\right)$ (the corresponding fixed internal ray connecting $\xi'$ to $\infty$ is shown in cyan). Here $\xi$ and $\xi'$ are the roots. The domain $V_n$ is shaded in grey. In this example, there are $4$ distinct accesses to $\xi$ and $3$ distinct accesses to $\xi'$ from within $V_n$. This means that $I_1 = [t_1', t_2]$ (in black) contains $4$ distinct points $t_1', y_1, y_2, t_2$ that are mapped by $\gamma_n$ to $\xi$; similarly, $I_2 = [t_2', t_1]$ (in blue) contains $3$ distinct points $t_2', y_1', t_1$ that are mapped to $\xi'$.}
\label{Fig:Possibilities}
\end{center}
\end{figure}  

\subsubsection{Identifying degree 1 interval.} 
\label{SSec:Deg1}
{A priori}, the $3m$ points $t_1, s_1, t_1', \ldots$ depend on $n$. But since the fixed points along $\partial V_n$ and $\partial V_{n+1}$ are the same, we can choose the boundary parametrization $\gamma_{n+1} \colon \Circle \to \partial V_{n+1}$ such that the corresponding map $\tau_n \colon \Circle \to \Circle$ that connects the boundary parametrizations $\gamma_{n+1}$ and $\gamma_n$ satisfies $\tau_n(s_i)=s_i$, $\tau_n(t_i)=t_i$ and $\tau_n(t_i') = t_i'$ for all $i$ (this can be accomplished by precomposing $\tau_n$ and $\gamma_{n+1}$ by the same circle homeomorphism). 

Defined this way, $\tau_n$ maps $[t_i, t_i']$ homeomorphically onto itself for each $i$. Going by induction in the same way, we define the mappings $\gamma_n \colon \Circle \to \partial V_n$ and $\tau_n \colon \Circle \to \Circle$ for all $n \ge n_a$. All these maps have the property that the points $s_i$, $t_i$, and $t'_i$ are fixed under the $\tau_n$, and are mapped to the same fixed points of $N_p$ ($\infty$ resp.\ roots) under all  $\gamma_n$.

\begin{claim}
\label{Claim:HomeoInterval}
There exists an $i$ such that $\tau_n|_{I_i}$ is an orientation-preserving homeomorphism of $I_i$ onto itself for all $n\ge n_a$.
\end{claim}
\begin{proof}
We will use the fact, established (in greater generality) in \cite[Lemma~3.3]{DMRS}, that the domain $V_{n+1}$ has $\deg \tau_n$ distinct accesses to $\infty$. Thus $\tau_n$ has degree $m$ for all $n\ge n_a$. 

Since $\tau_{n_a}$ restricted to any $I_i$ must fix the endpoints of $I_i$, its image either equals $I_i$ or all of $\Circle$. Each $s_k$ has one preimage at itself, and then at least one further preimage in each interval $I_i$ on which $\tau_{n_a}$ is not a homeomorphism. Since there are $m$ such intervals and $\deg \tau_n = m$, the claim follows for $n=n_a$.

Choose labels so that $i = 1$ is such that $\tau_{n_a}|_{I_i}$ is a homeomorphism. We now proceed by induction. 

Since $\tau_n|_{I_1}$ is a homeomorphism and $\gamma_n(t)=\infty$ only for $t=s_i$, \eqref{Eq:FuncRel} implies that the restriction $\gamma_{n+1}\colon I_1\to\partial V_{n+1}$ does not pass through a pole. It follows from Lemma~\ref{Lem:BasicProperties_V_n} that $\gamma_{n+2}\colon I_1\to\partial V_{n+2}$ cannot pass through a pole either, so $\tau_{n+1}|_{I_1}$ is a homeomorphism as well. This completes the claim.
\end{proof}

Again choose labels so that $I_1 = [t_1', t_2]$ is the degree one interval given by Claim~\ref{Claim:HomeoInterval}. Denote $\xi_1 := \gamma_n(t_1')$ and $\xi_2 := \gamma_n(t_2)$ the roots corresponding to the endpoints of $I_1$. Further introduce $\Gamma_1 := \gamma_n((s_1, t_1'])$ and $\Gamma_2 := \gamma_n([t_2,s_2))$; these are the fixed rays connecting $\xi_1$ resp. $\xi_2$ to $\infty$ within $\partial V_n$.

\subsubsection{Defining the bridge in the case $\xi_1 \neq \xi_2$.}  
\label{SSec:Br1}
Define $X_n$ to be the unique Jordan arc in $\gamma_n (I_1)$ connecting $\xi_1$ and $\xi_2$ (see Figure~\ref{Fig:Bridges}).  

In the case $\xi_1 = \xi_2 =: \xi$, the image $\gamma_n(I_1)$ might contain several loops connecting $\xi$ to itself (compare Figure~\ref{Fig:Possibilities}: there are three loops in the image of $I_1$). We need some additional analysis to pick a correct loop; this will be done over the next two steps. 

\begin{figure}[htbp]
\begin{center}
\includegraphics[width=0.8\textwidth, trim=30 30 30 30]{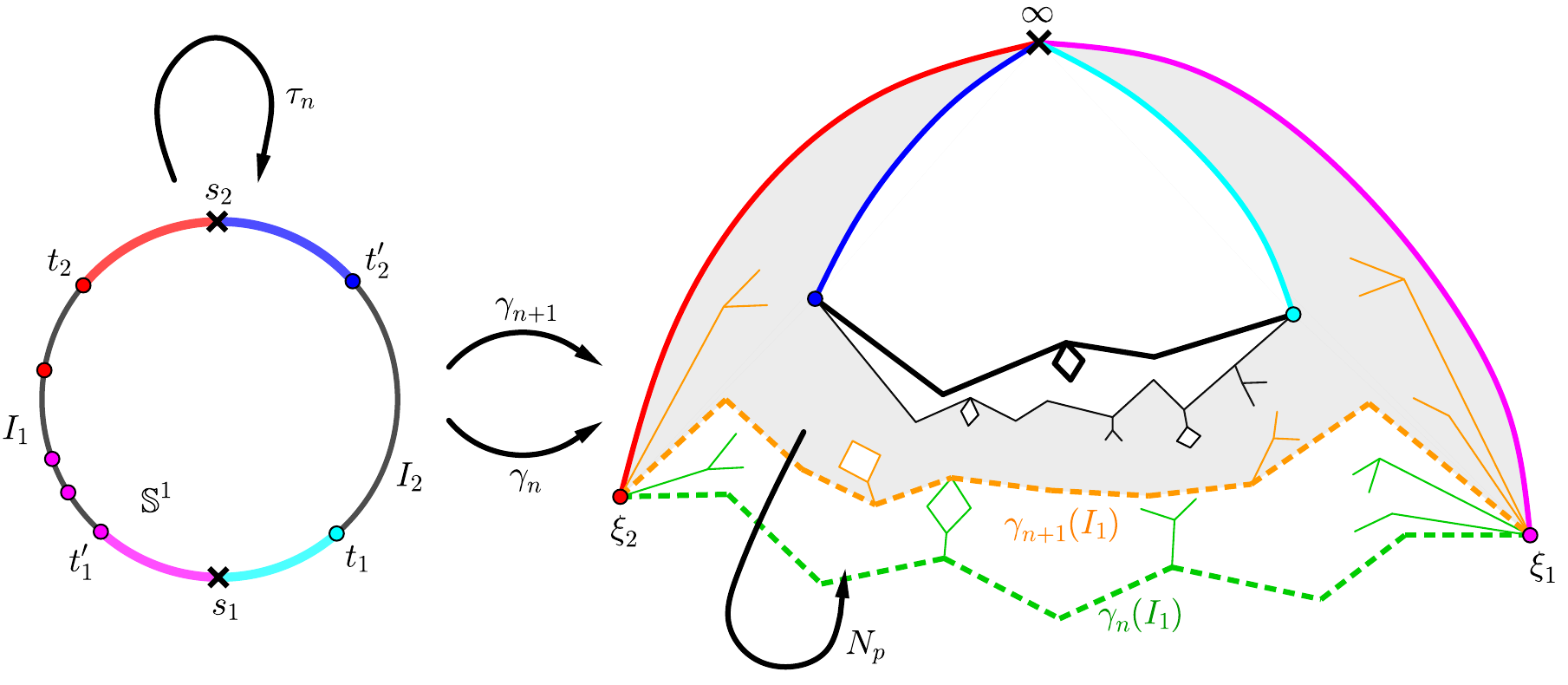}\hfill
\caption{Construction of bridges in the proof of Lemma~\ref{Lem:SweepOut}. The case $m=2$ and $\xi_1 \neq \xi_2$ is shown. The domain $V_{n+1}$ is depicted in grey, while the thick red, blue, black, magenta and green (both solid and dashed) lines bound the domain $V_n \supset V_{n+1}$. The image of the interval $I_1 = [t_1', t_2]$ under $\gamma_n$ is shown in green. The respective image for $\gamma_{n+1}$ in $\partial V_{n+1}$ are shown in orange. The bridge $X_n$ connecting two roots $\xi_1$ and $\xi_2$ is the dashed green line in $\gamma_n(I_1)$; the bridge $X_{n+1}$ is the dashed orange line.}
\label{Fig:Bridges}
\end{center}
\end{figure}  

\subsubsection{Mapping of accesses to roots in the image of degree 1 interval.}  
\label{SSec:Mapping}
Let $\xi \in \gamma_n(I_1)$ be a root. Then $V_n \cap U_\xi$ is a disjoint union of sectors $B_{n,i} = B_{n,i}(\xi)$, each bounded by a pair of (pre-)fixed internal rays and a piece of the boundary of $U_\xi$ (see Figure~\ref{Fig:MappingOfAccesses}). There is a one-to-one correspondence between the sectors $B_{n,i}$ and the accesses to $\xi$ within $V_n$ (see Definition~\ref{Def:SimpleAccess}); to simplify terminology, we will call $B_{n,i}$ an access (to $\xi$). 

For each $B_{n,i}$ there exists a unique point $v_{n,i} = v_{n,i}(\xi) \in I_1$ and a non-degenerate interval $J_{n,i} = J_{n,i}(\xi) \subset [s_1, t_1'] \cup I_1 \cup [t_2,s_2]$ such that $\gamma_n(v_{n,i}) = \xi$ and $\gamma_n(J_{n,i}) = \partial B_{n,i} \cap U_\xi$. Choose the labeling so that $v_{n,1} < v_{n,2} < \ldots < v_{n,s}$, where $s = s(\xi)$ is the number of accesses to $\xi$ in $V_n$. 

A priori, the number $s$ depends not only on $\xi$, but also on $n$. Let us show that for all $n$ large enough this dependence vanishes.

\begin{claim}
\label{Claim:AccessesConstant}
There exists $n'_a \ge n_a$ such that for each $n \ge n_a'$ the number of accesses to any given root in $\gamma_n(I_1)$ is constant.  
\end{claim}

\begin{proof}
Let $t \in I_1$ be a point such that $\gamma_{n+1}(t) = \xi$ is a root; using \eqref{Eq:FuncRel}, $\gamma_n(\tau_n(t)) = \xi$. But $\tau_n|_{I_1}$ is a homeomorphism onto $I_1$, thus $\tau_n(t) \in I_1$. Therefore, $\xi \in \gamma_n(I_1)$. Since both $t$ and $\tau_n(t)$ correspond to an access to $\xi$ in $V_{n+1}$, resp.\ $V_n$, it follows that $|\mathcal A_\xi(V_n)| \ge |\mathcal A_\xi(V_{n+1})|$. The sequence $(|\mathcal A_\xi(V_n)|)_{n \ge n_a}$ is thus non-increasing, and hence is constant starting some $n_a' \ge n_a$. Increase $n_a'$ if necessary so that the latter conclusion holds for all roots in $\gamma_n(I_1)$.
\end{proof}

By Claim~\ref{Claim:AccessesConstant}, $\tau_n(v_{n+1,i}) = v_{n,i}$ and $\tau_n(J_{n+1,i}) = J_{n,i}$ for each $i \in \{1, \ldots, s\}$ and every $n \ge n_a'$. Since $\gamma_k(J_{k,i}) = \partial B_{k,i} \cap U_\xi$, $k \in \{n, n+1\}$ and $U_\xi$ is invariant under $N_p$, we also conclude that $N_p(B_{n+1,i}) = B_{n,i}$ for all $i \in \{1, \ldots, s\}$ and $n \ge n_a'$. Here the data depends on $\xi$.

\begin{figure}[htbp]
\begin{center}
\includegraphics[width=0.35\textwidth, trim=10 0 10 18, clip]{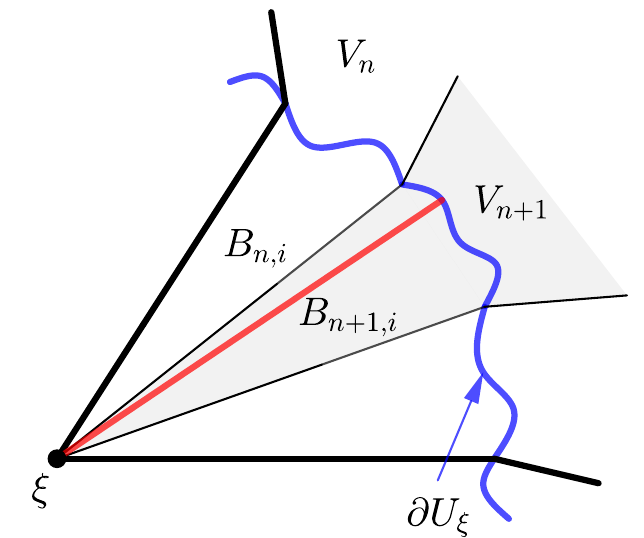}\hfill
\caption{This picture shows the situation locally around the root $\xi$ focusing on the accesses $B_{n+1,i}$ and $B_{n,i} = N_p(B_{n+1,i})$ in the immediate basin $U_\xi$. In the depicted situation, $B_{n+1,i} \subset B_{n,i}$, and hence $\ovl{B_{n+1,i}} \cap \ovl{B_{n,i}}$ contains a fixed ray of $N_p$ (shown in red).}
\label{Fig:MappingOfAccesses}
\end{center}
\end{figure}

\subsubsection{Inclusion of accesses to roots in the image of degree 1 interval.}  
\label{SSec:Inclusion}
Let $\xi \in \gamma_n(I_1)$ be a root, and $n$ be at least $n'_a$. Since $V_{n+1} \subset V_n$, it follows that for each $i$ there exists a $j$ such that $B_{n+1, j} \subset B_{n,i}$. Introduce \emph{the inclusion map} $T_n \colon \{1, \ldots, s\} \to \{1, \ldots, s\}$ defined as $T_n(i) = j$. Because the homeomorphism $\tau_n$ is orientation-preserving, the map $T_n$ is order-preserving.

The following claim summarizes some properties of such maps. The meaning of the second assertion of the claim will be clear a bit later: this will be the property that will allow us to construct the bridge in the case $\xi_1 = \xi_2$.

\begin{claim}
\label{Claim:T}
If $s \ge 2$ and $T \colon \{1, \ldots, s\} \to \{1, \ldots, s\}$ is an order-preserving map, then
\begin{enumerate}
\item
\label{It:FP}
$T$ has at least one fixed point;
\item
\label{It:WB}
if, additionally, $s > 2$ and $1, s$ are the only fixed points of $T$, then there exists a unique $j \in \{2, \ldots, s-1\}$ such that $T(j) \le j$ and $T(j+1) \ge j+1$ (with at least one inequality being strict).
\end{enumerate}
\end{claim}

\begin{proof}
Let us show that $T$ has a fixed point. Indeed, if $T(1) = 1$, then we are done. Otherwise, $T(1) > 1$, and if $\ell := \max\{ i \colon T(i) > i\}$, then $\ell+1$ is a fixed point.

Suppose $s > 2$ and $1,s$ are the only fixed points of $T$. For the remaining $s-2$ integers, we have either $T(i)>i$, or $T(i) < i$. If $T(2) > 2$, then $\ell = s-1$, because otherwise $\ell+1$ is a fixed point of $T$ distinct from $1$ and $s$; in this case, $j=1$. If $T(2) = 1$, then by a similar reasoning one can see that $j = \max\{ i \colon T(i) < i\}$.
\end{proof}

Coming back the the Newton setting, the existence of a fixed point of an inclusion map yields the following result.  

\begin{claim}
\label{Claim:FixedRay}
If $T_n(i) = i$, that is $B_{n+1, i} \subset B_{n,i}$, for some $i \in \{1, \ldots, s\}$, then $\ovl{B_{n+1,i}} \cap \ovl{B_{n,i}}$ contains a fixed ray of $N_p$ (see Figure~\ref{Fig:MappingOfAccesses}).  
\end{claim}

\begin{proof}
Consider a Riemann map $\phi \colon U_\xi \to \disk$ with $\phi(\xi) = 0$. It induces a map $h := \phi \circ N_p \circ \phi^{-1} \colon \disk \to \disk$ of the form $h(z) = \lambda z^k$ with $|\lambda|=1$ and $k \ge 2$. Any edge in $\partial V_n$ terminating at $\xi$ corresponds under $\phi$ to a radius in $\disk$. Therefore, $S:=\phi(B_{n,i})$ and $S':=\phi(B_{n+1,i})$ are two open sectors, each bounded within $\disk$ by two radii, satisfying $S' \subset S$ and $h(S') = S$. The action of $h$ on $\partial \disk$ is an expanding circle endomorphism of degree $k \ge 2$, and thus $\ovl{S'}$ contains a radius that is fixed under $h$. Therefore, $\ovl{B_{n+1,i}} \cap \ovl{B_{n,i}}$ contains a ray from $\xi$ that is fixed under $N_p$.
\end{proof}

Here is an immediate corollary to the previous two claims.

\begin{claim}
\label{Claim:NoOtherRoots}
For every $n \ge n_a'$, the image $\gamma_n(I_1)$ does not contain a root of $p$ other than $\xi_1$ and $\xi_2$. 
\end{claim}

In this claim we do not exclude a possibility that the curve $\gamma_n \colon \inter I_1 \to \partial V_n$ passes through $\xi_1$ or $\xi_2$ (compare Figure~\ref{Fig:Bridges}: $\inter I_1$ contains $3$ pink points that are mapped to $\xi_1$ and $2$ red points that are mapped to $\xi_2$).

\begin{proof}
If $\xi \in \gamma_n(I_1)$ is a root and $\xi \not \in \{\xi_1, \xi_2 \}$, then by Claim~\ref{Claim:T}~\eqref{It:FP} there exists an $i$ so that $T_n(i) = i$, where $T_n$ is the inclusion map of the accesses to $\xi$. Hence, by Claim~\ref{Claim:FixedRay} there exists a fixed ray in $\bigcap_{n \ge n_a'} \ovl{B_{n,i}}$. This fixed ray lies in the channel diagram $\Delta$ and connects $\xi$ to $\infty$, and hence it must be either in each $\partial V_n$, or must ``cut through'' and disconnect each $V_n$. Both options are impossible by construction.
\end{proof}

The following two claims establish some further properties of the inclusion map. 

\begin{claim}
\label{Claim:DoNotDepend}
The inclusion map $T_n$ defined for $\xi \in \{\xi_1, \xi_2\}$ does not depend on $n$ (but of course depends on the root). 
\end{claim}

\begin{proof}
We want to show that if $B_{n+1, j} \subset B_{n,i}$, i.e.\ $T_n(j) = i$, then $B_{n+2, j} \subset B_{n+1,i}$, that is $T_{n+1}(j) = i$. The claim will then follow by induction.

Let $f$ be the branch of $N_p^{-1}$ that maps $B_{n,i}$ over $B_{n+1,i}$. This inverse branch maps $B_{n+1, j} \subset B_{n,i}$ over one of the accesses to $\xi$ in $V_{n+2}$. Because $B_{n+2,j}$ is the unique access such that $N_p(B_{n+2,j}) = B_{n+1,j}$, it follows $f(B_{n+1,j}) = B_{n+2,j}$. Applying $f$ to the nested pair $B_{n+1, j} \subset B_{n,i}$ we conclude $B_{n+2,j} \subset B_{n+1,i}$.
\end{proof}

By Claim~\ref{Claim:DoNotDepend}, we can now speak about \emph{the} inclusion map for each of the roots $\xi_1$ and $\xi_2$. Another consequence of Claim~\ref{Claim:FixedRay} is that, when $\xi_1 \neq \xi_2$, each of these inclusion maps has a unique fixed point ($1$ for $\xi_1$, and $s$ for $\xi_2$); they correspond to the access adjacent to the fixed ray $\Gamma_1$ resp.\ $\Gamma_2$. For the rest of the points these maps are strictly monotone (decreasing for $\xi_1$ and increasing for $\xi_2$). Similarly, when $\xi_1 = \xi_2 =: \xi$, the inclusion map $i \mapsto T(i)$ defined for $\xi$ has exactly two fixed points, namely, $1$ and $s$, and if $s>2$, then it is strictly decreasing for $i \le j$ and strictly increasing for $i \ge j+1$, where $j$ is given by Claim~\ref{Claim:T}~\eqref{It:WB}. These observations justify the following claim. 

\begin{claim}
\label{Claim:Monotonicity}
If $\xi \in \{\xi_1, \xi_2\}$ is a root, $s$ is the number of accesses to $\xi$ in $V_n$ ($n \ge n_a'$), and $T \colon \{1, \ldots, s\} \to \{1, \ldots, s\}$ is the inclusion map for $\xi$, then for every $i \in \{1, \ldots, s\}$
\begin{enumerate}
\item
if $\xi_1 \neq \xi_2$,
\begin{equation*}
T^{\circ s}(i) = \left\{
\begin{aligned}
&1, \quad \text{ if } \xi = \xi_1,\\
&s, \quad \text{ if } \xi = \xi_2;
\end{aligned}
\right.
\end{equation*}
\item
if $\xi_1 = \xi_2$,
\begin{equation*}
T^{\circ s}(i) = \left\{
\begin{aligned}
&1, \quad \text{ if } i \le j,\\
&s, \quad \text{ if } i \ge j+1,
\end{aligned}
\right.
\end{equation*}
where $j$ is given by Claim~\ref{Claim:T}~\eqref{It:WB}. \qed
\end{enumerate}
\end{claim}

\subsubsection{Defining the bridge in the case $\xi_1 = \xi_2$.} 
\label{SSec:Br2}
Denote $\xi_1 = \xi_2 =: \xi$, $s := |\mathcal A_{\xi}(V_n)|$, and let $T$ be the inclusion map for $\xi$; as usual, we assume $n \ge n_a'$.

\begin{claim}
\label{Claim:s2}
The two edges in $\gamma_{n}(I_1)$ closest to $\Gamma_1$ and $\Gamma_2$ are different for all $n > n_a'$.
\end{claim}

\begin{proof}
If this is not the case, then $s=2$ and $\partial V_n \cap U_\xi = \Gamma_1 \cup \Gamma_2 \cup e_n$ for all $n \ge n_a'$, where $e_n$ is some pre-fixed internal ray. In particular, $\Delta_n \cap \ovl V_n \cap U_\xi = \Gamma_1 \cup \Gamma_2 \cup e_n$ for all $n \ge n_a'$. But $\xi$ is a critical point, and hence the number of edges in $\Delta_n$ within each of the components of $U_\xi \sm (\Gamma_1 \cup \Gamma_2)$ is growing with $n$ (one of these components is equal to $V_n \cap U_\xi$). This contradicts our assertion that $\partial V_n$ intersects $U_\xi$ at exactly three edges. 
\end{proof}

We are now in position to define the bridges when the roots coincide. 

If $s=2$ for all $n \ge n_a'$, then $\gamma_{n+1}(\inter I_1)$ contains a Jordan loop connecting $\xi$ to itself. The existence of such a loop follows from Claim~\ref{Claim:s2}: it starts and ends with two different edges in $\gamma_{n+1}(I_1) \cap U_{\xi}$.  This is the required bridge $X_{n+1}$, and it is uniquely defined.

If $s > 2$, then by Claim~\ref{Claim:T}~\eqref{It:WB} there exists a unique $j \in \{2, \ldots, s-1\}$ such that $T(j) \le j$ and $T(j+1) \ge j+1$, with at least one inequality being strict. This inclusion behavior is only possible when the image under $\gamma_{n+1}$ of $[v_{n+1, j}, v_{n+1,j+1}]$ surrounds at least one access to $\xi$ within $V_n$ (see Figure~\ref{Fig:T}). Therefore, $\gamma_{n+1}((v_{n+1, j}, v_{n+1,j+1}))$ contains a Jordan loop connecting $\xi$ to itself. This loop is the bridge $X_{n+1}$; again, this bridge is uniquely defined.

\begin{figure}[htbp]
\begin{center}
\includegraphics[width=0.91\textwidth, trim=10 9 5 10, clip]{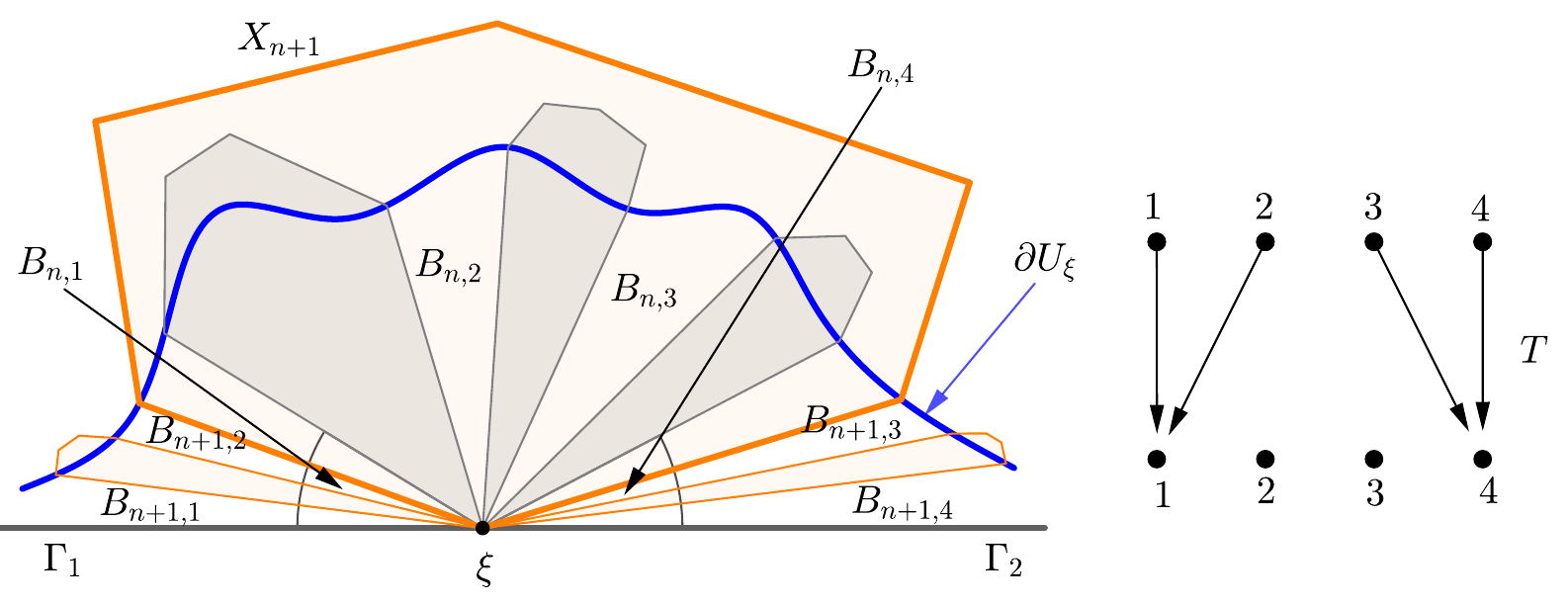}\hfill
\caption{Inclusion of the accesses governed by the inclusion map $T$. The case $s=4$ and $j=2$ is shown. The bridge $X_{n+1}$ is marked with thick orange line. The complement of $V_n$ is shaded in grey, the complement of $V_{n+1}$ is shaded in light orange.}
\label{Fig:T}
\end{center}
\end{figure} 
 
\subsubsection{Properties of bridges.}  
\label{SSec:Props}
In Subsections~\ref{SSec:Br1} and \ref{SSec:Br2}, the bridge $X_n$ connecting $\xi_1$ to $\xi_2$ is defined for all $n > n_a'$. We now derive some properties of bridges; these properties will yield Lemma~\ref{Lem:SweepOut} in the next (final) step of the proof.

\begin{claim}
\label{Claim:BridgesMapping}
There exists $k_a > n_a'$ such that $N_p (X_{n+1}) = X_n$ and $N_p \colon X_{n+1} \to X_n$ is a homeomorphism for every $n \ge k_a$. 
\end{claim}

\begin{proof}
By construction of bridges, $X_n \subset N_p(X_{n+1})$ for every $n > n_a'$. Indeed, if $\xi_1 \neq \xi_2$, then $N_p (X_{n+1})$ is a subset of $\gamma_n(I_1)$ that connects $\xi_1$ and $\xi_2$ within $\partial V_n \cap \C$. Hence, $X_n \subset N_p(X_{n+1})$ by definition of the bridge in the case $\xi_1 \neq \xi_2$. If $\xi_1 = \xi_2$, then the claim follows from Claim~\ref{Claim:DoNotDepend} and the fact that the bridge for $\xi_1 = \xi_2$ was constructed uniquely only using the properties of the inclusion map.

Call the \emph{length} of the bridge the number of edges of the Newton graph this bridge contains. Since $X_n \subset N_p(X_{n+1})$, the lengths of the bridges are non-decreasing as $n \to \infty$. But for every $n > n_a'$ the length of $X_n$ is bounded above by the number of intervals in $I_1$ that are mapped by $\gamma_n$ to edges of $\partial V_n$ (restricted to either of these intervals $\gamma_n$ is a homeomorphism); this number is independent of $n$ because $\tau_n$ is a homeomorphism of $I_1$. Therefore, there exists $k_a > n_a'$ such that for all $n \ge k_a$ the length of $X_n$ is constant; in particular, $N_p(X_n)$ is a Jordan curve. For such $n$ it then follows that $N_p(X_{n+1}) = X_n$ and $N_p \colon X_{n+1} \to X_n$ is a homeomorphism.  
\end{proof}

Define $L_n := X_n \sm (U_{\xi_1} \cup U_{\xi_2})$, $e_n^1 := X_n \cap U_{\xi_1}$, and $e_n^2 := X_n \cap U_{\xi_2}$. The first set is the central part of the bridge, that is the bridge without its terminal edges $e_n^1$ and $e_n^2$. Note that the central part can be just a point. Since the immediate basins are invariant under $N_p$, Claim~\ref{Claim:BridgesMapping} implies that the central parts of bridges are mapped by $N_p$ homeomorphically over central parts, and terminal edges are mapped over terminal edges.

\begin{claim}
\label{Claim:Extension}
The homeomorphism $N_p \colon L_{n+1} \to L_n$ admits a biholomorphic extension to some neighborhood of $L_{n+1}$ for every $n \ge k_a$ (possibly subject to increasing $k_a$).
\end{claim}

Note that such an extension is clearly impossible for the whole bridge because of the critical points at its ends.  

\begin{proof}
By Claim~\ref{Claim:BridgesMapping}, the map $N_p \colon X_{n+1} \to X_n$ sends edges to edges and vertices to vertices in a one-to-one fashion and the lengths of the bridges are constant for every $n \ge k_a$. Using Claim~\ref{Claim:NoOtherRoots} we then conclude that every particular bridge can intersect only finitely many other bridges. Therefore, by possibly increasing $k_a$ to get rid of the critical points that are not fixed under $N_p$ but still lie in $L_{n+1}$ for some initial values of $n$, we can guarantee that $L_{n+1}$ contains no critical points of $N_p$ for every $n \ge k_a$. The claim follows.   
\end{proof}

\begin{claim}
\label{Claim:Angles}
For every $i \in \{1, 2\}$, the angle between $\Gamma_i$ and $e_n^i$ within $V_n$ tends to zero as $n \to \infty$.
\end{claim}

\begin{proof}
Let us show this for $i = 1$. For each $n \ge k_a$, consider a pair of accesses to $\xi$ within $V_n$: the one that is adjacent to $\Gamma_1$, call it $B_n$, and the one that is adjacent to $e_n^1$, call it $B_n'$. Both of these accesses are components of $U_{\xi_1} \cap V_n$, and it is possible that $B_n = B_n'$. 

On the immediate basin $U_{\xi_1}$ the Newton map $N_p$ is conjugate to the power map $z \mapsto \lambda z^k$, $k \ge 2$, $|\lambda|=1$ via a Riemann map $\phi \colon U_{\xi_1} \to \disk$, $\phi(\xi) = 0$. Assume that the Riemann map is rotated so that the ray $\Gamma_1$ corresponds to the ray $[0,1)$ in the disk. Denote by $\alpha_n$ and $\alpha_n'$ the angles of the sectors $\phi(B_n)$ and $\phi(B_n')$ in the disk (compare the proof of Claim~\ref{Claim:FixedRay}). 

Our goal is to show that $\alpha'_n \to 0$ when $n \to \infty$. Since $N_p(B_{n+1}) = B_n$ and $\ovl{B_{n+1}} \cap \ovl{B_n} = \Gamma_1$, we have $\alpha_{n+1} = \alpha_n / k$, and hence $\alpha_n \to 0$ as $n \to \infty$. By Claim~\ref{Claim:Monotonicity}, $B_n' \subset B_{n - s}$ for every $n \ge k_a + s$, where $s$ is the number of accesses to $\xi_1$. Therefore, $\alpha'_n \le \alpha_{n-s}$, and hence $\alpha_n' \to 0$ as $n \to \infty$ as well. 
\end{proof}

\begin{claim}
\label{Claim:SA}
As  $n\to\infty$, the bridges $X_n$ tend to $\ovl \Gamma_1 \cup \ovl \Gamma_2$ as a set.
\end{claim} 
\begin{proof}
This was shown in \cite[Theorem~3.4]{DMRS}, Claim 3. The idea is that the terminal edges $e_n^1$ and $e_n^2$ of $X_n$ run within immediate basins and connect the roots to pre-poles, say $r_n^1$ and $r_n^2$. Claim~\ref{Claim:Angles} implies that these points converge to $\infty$ as $n \to \infty$. It also follows that the terminal edges $e_n^i$ converge to $\Gamma_i$ as sets. The rest of the bridges $X_n$, i.e.\ the central parts $L_n$, connect two points close to $\infty$ and have short spherical lengths (the central parts $L_n$ have constant hyperbolic lengths in domains that are close to $\infty$), and this implies the claim (it is here that we use Claim~\ref{Claim:Extension}). Details can be found in \cite{DMRS}.
\end{proof}

\begin{proof}[Proof of Lemma~\ref{Lem:SweepOut}]
Conclusion~\eqref{It:ML1} of the lemma follows from Subsections~\ref{SSec:Br1} and \ref{SSec:Br2}. Conclusion~\eqref{It:ML2} follows from Claim~\ref{Claim:BridgesMapping}. Finally, conclusion~\eqref{It:ML4} follows from Claim~\ref{Claim:SA}. 
\end{proof}

\subsection{Puzzle construction for Newton maps of polynomials}
\label{SSec:Puz}

In this subsection we will present the construction of puzzles for an iterate of an (attracting-critically-finite) Newton map. The first step in this construction is the Newton graph $\Delta_n$ itself: it is a forward invariant graph, and thus might serve to define a Markov partition of $\Cc$. However, the components of $\Cc \sm \Delta_n$ (called puzzle pieces) are not necessarily Jordan disks, which is a problem. It is desirable for puzzle pieces to be Jordan disks in order to apply the standard theory of polynomial-like maps. Proposition~\ref{Prop:Circles} will help us resolve this problem. We will refine a Newton graph by adding to it all separating circles (the existence of which is proven in the proposition). The details of the construction are as follows (compare Figure~\ref{Fig:Circles}).

\begin{figure}[htbp]
\begin{center}
\includegraphics[scale=1.4, trim=20 20 20 20]{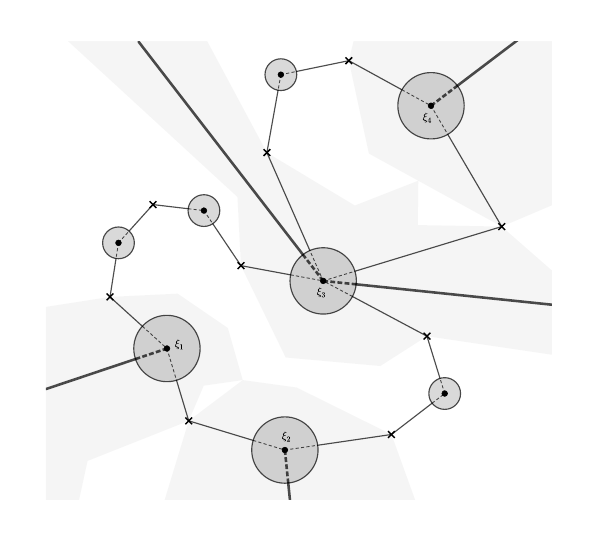}
\caption{A schematic drawing of the dynamical plane of a degree 4 Newton map. A part of the Newton graph is shown. Vertices that belong to the Fatou set are marked with circles; vertices that belong to the Julia are marked with crosses. Thick lines describe the channel diagram $\Delta$. The lightly shaded regions represent the immediate basins of roots. The thin segments, constructed as iterated preimages of $\Delta$, form two topological circles passing through the roots $\xi_1, \ldots, \xi_4$ and are guaranteed to exist by Proposition~\ref{Prop:Circles}. All black lines except circles form $\Deltap_0$. The big grey disks are bounded by suitably chosen equipotentials in the immediate basins; the small grey disks are iterated preimages of the big ones intersecting the graph $\Deltap_0$. All solid black lines comprise $\Deltat_0$.}
\label{Fig:Circles}
\end{center}
\end{figure} 

Let $\Deltap$ be the union of $\Delta$ and all circles in all the components of $\Cc \sm \Delta$  given by Proposition~\ref{Prop:Circles} for the least possible value of $K$, and let $\Deltap_n$ to be the component of $N_p^{-n}(\Deltap)$ containing $\infty$ (we assume $\Deltap_0:= \Deltap$).

A point $w$ is a \textit{(pre-)pole at level $l\ge 1$} if $w \in N_p^{-l}(\infty)$ with minimal $l$. Let $X$ be the union of all the circles given by Proposition~\ref{Prop:Circles} for the least possible value of $K$, and let $\nu$ be the largest level of (pre-)poles in $X$. Recall that $N$ is the minimal level of the Newton graph such that $\Delta_n$ contains all the poles of $N_p$ for all $n \geqslant N$ (see Theorem~\ref{Thm:PolesInGraph}).  The following lemma summarizes the essential properties of $\Deltap_n$ in terms of $N$ and $\nu$.

\begin{lemma}[Properties of $\Deltap_n$]
\label{Lem:PropertiesDeltap}
In the notations above, 
\begin{enumerate}
\item
\label{It:PDP1}
$\Delta_n \subset \Deltap_n$ for all $n$;

\item
\label{It:PDP2}
$\Deltap_{(i-1)( N - 1 + \nu)} \subset \Deltap_{i(N - 1 + \nu)}$ for every positive integer $i$;

\item
\label{It:PDP3}
$N_p^{-l}(\Deltap_n) = \Deltap_{n+l}$ for all positive integers $n$ and $l$ such that $n \geqslant N - 1 + l$;

\item
\label{It:PDP4}
for every $n \geqslant 0$ each component of $\Cc \sm \Deltap_n$ is a disk that might be pinched only at (finitely many) iterated preimages of roots.
\end{enumerate}
\end{lemma} 

The last condition says that each component is a disk, and that it may fail to be a Jordan disk only because the boundary runs finitely many times through the same root or one of its preimages (so that the component intersects a neighborhood of the root or its preimage in finitely many sectors).

\begin{proof}
(\ref{It:PDP1}): Since $\Delta \subset \Deltap$, the first property immediately follows from the definition of $\Deltap_n$.

(\ref{It:PDP2}): Observe that since $\Delta_N$ contains all the poles of $N_p$, inductively $\Delta_{N-1+\nu}$ contains all the (pre-)poles of $N_p$ at level at most $\nu$ (see Corollary~\ref{Cor:PrepolesInNewtonGraph}), thus $\Delta_{N-1 + \nu} \supset \Deltap$. Therefore, by definition of $\nu$ and using (\ref{It:PDP1}), we obtain (\ref{It:PDP2}) for $i = 1$. The rest follows by induction.

(\ref{It:PDP3}): If $n \geqslant N - 1 + l$, then $\Delta_n$ (and hence $\Deltap_n$, by (\ref{It:PDP1})) contains all the (pre-)poles at level at most $l$. However, every component of $N_p^{-l}(\Deltap_n)$ must contain a prepole of level $l$ (since $\Deltap_n \supset \Delta$; see Corollary~\ref{Cor:PrepolesInNewtonGraph}), and therefore, every such component is connected to $\Deltap_n$. From definition of $\Deltap_n$ it then follows that $N_p^{-l} (\Deltap_n) = \Deltap_{n+l}$. 

(\ref{It:PDP4}): To begin with, by construction, every component of $\Cc \sm \Deltap_0$ is an open topological disk. Assume $n > 0$. Every (pre-)pole on the graph $\Deltap_n$ is either an iterated preimage of a vertex on the circle given by Proposition~\ref{Prop:Circles}, or an iterated preimage of the vertex $\infty$ on the channel diagram $\Delta$. Since, by Proposition~\ref{Prop:Circles}, the circle in $\Cc \sm \Delta$ is disjoint from the postcritical set of $N_p$, it follows that the points of the first type are always uniquely accessible from every domain with such points on the boundary. Every access to each (pre-)pole of the second type must be separated from any another access to the same (pre-)pole by a pair of preimages of the corresponding bridges in $\Deltap_0$ (these bridges are ``circular arcs'' connecting the corresponding roots in $\Delta$ within the circle given by Proposition~\ref{Prop:Circles}). Therefore, each (pre-)pole of the second type cannot be multiply accessible. The claim follows. 
\end{proof}

Lemma~\ref{Lem:PropertiesDeltap} gives us almost all we need to construct renormalization domains for (an iterate of) the Newton map $N_p$. However, by property (\ref{It:PDP4}) in Lemma~\ref{Lem:PropertiesDeltap}, a component of $\Cc \sm \Deltap_n$ can be pinched at finitely many iterated preimages of the roots. We want to avoid such pinchings by considering a \textit{truncation} of $\Deltap_n$. This a standard procedure that is carried out as follows. 

For every root $\xi_i$ of $N_p$ take an arbitrary but fixed equipotential $E(\xi_i)$ in the immediate basin $U_i$. By definition, $E(\xi_i)$ is the preimage of the circle $\{z \colon |z| = r\}$ (for some $r\in(0,1)$) under any Riemann map $\phi \colon U_i \to \disk$ with $\phi(\xi_i) = 0$.

We extend this equipotential to preimages of roots as follows. Consider a point $z_{i,k} \in \Deltap_0$ with $N_p^{\circ k}(z_{i,k}) = \xi_i$ and minimal $k$. Let $E(z_{i,k})$ be the component of $N_p^{-k}(E(\xi_i))$ in the Fatou component containing $z_{i,k}$; by construction, $E(z_{i,k})$ is a smooth closed curve separating $z_{i,k}$ from the boundary of its Fatou component and hence from $\infty$. Let $E_0$ be the union of all $E(z_{i,k})$, where $z_{i,k}$ runs over all  preimages of roots in  $\Deltap_0$ (this includes all the finite fixed points of $N_p$), and let $Q_0$ be the unique unbounded component of $\Cc \sm E_0$ (see Figure~\ref{Fig:Circles}). Set
$$
\Deltat_0 := \left(\Deltap_0 \cap Q_0 \right) \cup E_0.
$$

\begin{definition}[Truncated puzzle of depth $n$]
\label{Def:TruncPuzzle}
The \textit{truncated puzzle of depth $n$} (with $n \geqslant 0$) is the component $\Deltat_n$ of $N_p^{-n}(\Deltat_0)$ that contains $\infty$.
\end{definition}

Using truncated puzzles we can finally define Newton puzzles for iterates of $N_p$.

\begin{definition}[Puzzles and puzzle pieces for iterate of Newton map]
\label{Def:Puzzles}
Set $M := N - 1 + \nu$. For $n \geqslant 0$, the \emph{puzzle of depth $n$} is the finite graph $\Deltat_{(n+2)M}$. A \emph{puzzle piece of depth $n$}, denoted as $P_n$, is the closure of a component of $\Cc \sm \Deltat_{(n+2) M}$ that intersects the Julia set of $N_p$. 

For every point $x$ which is not attracted to one of the roots of $p$ and for a given depth $n$, define $P_n(x)$ to be the union of all puzzle pieces of depth $n$ containing $x$. 
\end{definition}

From the definition above it is clear that if $x$ is not $\infty$ or a (pre-)pole, then $P_n(x)$ is a unique puzzle piece of depth $n$ containing the point $x$. Otherwise, $P_n(x)$ is a finite union of puzzle pieces with $x$ as their common boundary point.

\begin{definition}[Fiber]
Let $x \in \Cc$ be a point that is not attracted to any of the roots of $p$; the set 
$$
\fib(x) := \bigcap_{n \geqslant 0} P_n(x)
$$
is called the \emph{fiber} of $x$ (with respect to the partition given by the set of puzzles $(\Deltat_{(n+2)M})_{n=0}^\infty$).
\end{definition}

The following theorem summarizes the constructions we have done so far, and establishes properties of puzzle pieces and fibers for general (attracting-critically-finite) Newton maps. This is a more precise version of Theorem~\ref{Thm:B}, restricted to the attracting-critically-finite case. 
We view this theorem as one of the main results of the paper because it provides the foundation for all results on rigidity and local connectivity of Newton maps. 

\begin{theorem}[Newton puzzles for Newton maps]
\label{Thm:PuzPr}
Every attracting-critically-finite Newton map $N_p$ has an iterate $g := N_p^{\circ M}$ for which there exists a puzzle partition of any given depth $n$ with the following properties:
\begin{enumerate}
\item
\label{It:PuzPr1}
every puzzle piece is a closed Jordan disk;

\item
\label{It:PuzPr2}
any two puzzle pieces are either nested or have disjoint interiors; in the former case, the puzzle piece of larger depth is contained in the puzzle piece of smaller depth;  

\item
\label{It:PuzPr3}
if $x$ is a point that is not attracted to a root of $p$, then for all $n, k \geqslant 0$ the map
$$
g^{\circ k} \colon \inter P_{n+k}(x) \to \inter P_n (g^{\circ k}(x))
$$
is a branched covering;

\item
\label{It:PuzPr4}
if $x$ is $\infty$, a pole or a prepole, then $\fib(x) = \{x\}$.

\item
\label{It:PuzPr5}
if $x$ is not eventually mapped to $\infty$ or attracted to one of the roots of $p$, then the fiber $\fib(x)$ is a closed connected set such that $\fib(x) \subset \inter P_n(x)$ for every $n \geqslant 0$.
\end{enumerate}
\end{theorem}
\begin{proof}

The first part of the theorem is a summary of Definition~\ref{Def:Puzzles}: for $M = N-1+\nu$ the puzzle partition for $g$ of depth $n$ is given by the graph $\Deltat_{(n+2)M}$. Let us show that the listed properties hold true.

By Lemma~\ref{Lem:PropertiesDeltap}~(\ref{It:PDP4}), each component of the complement of $\Deltap_n$ is a topological disk with Jordan boundary except at finitely many ``pinching'' points that are iterated preimages of the roots. By construction of $\Deltat_n$, any such pinching point must be separated from $\infty$ by the preimage of the curve in $E_0$, and thus must be non-accessible from any puzzle piece of depth $n$. This proves property (\ref{It:PuzPr1}) of the theorem.

Property (\ref{It:PuzPr2}) was encoded in the definition of a puzzle piece and follows from Lemma~\ref{Lem:PropertiesDeltap}~(\ref{It:PDP2}), which essentially says that every puzzle of depth $n+1$ contains the puzzle of depth $n$ for each $n \ge 0$ (modulo the truncation in the basins of roots that respects this inclusion). 

Similarly, property (\ref{It:PuzPr3}) follows from Lemma~\ref{Lem:PropertiesDeltap} (\ref{It:PDP3}). Indeed, the latter property implies that 
\[
N_p^{-M}\left(\Deltap_{(n+2)M}\right) = \Deltap_{(n+2)M + M} = \Deltap_{(n+3)M}
\] 
because $(n+2)M \ge N-1 + M$, or equivalently $(n+1) (N - 1 +\nu) \ge N-1$, is true for every $n \ge 0$. Hence, the full preimage under the map $g = N_p^{\circ M}$ of the puzzle of depth $n \ge 0$ is the puzzle of depth $n+1$.  

Property (\ref{It:PuzPr4}) is a corollary of Proposition \ref{Prop:Circles}. Indeed, this proposition tells us that $P_0(\infty)$ contains no critical values of $N_p$, and hence the same holds for all $P_n(\infty)$ by the nesting property.
Since $\partial P_0(\infty)$ intersects $J(N_p)$ only at poles or prepoles, and thus contains no periodic points, there exists a $j > 0$ so that $P_{j}(\infty) \subset \inter P_0(\infty)$. The claim $\fib(\infty) = \{\infty\}$ now follows by the Schwarz lemma using the fact that $N_p \colon \inter P_1(\infty) \to \inter P_0(\infty)$ is a univalent map. 

Let us now establish property (\ref{It:PuzPr4}) for poles and prepoles. By the same argument as in the previous paragraph, for a given $n \ge 0$ there exists a $k >0$ so that $P_{n+k}(\infty) \subset \inter P_n(\infty)$. Pick the smallest possible value of $k = k(n)$, and set $A_n(\infty) := \inter P_{n+k}(\infty) \sm P_n (\infty)$ to be the corresponding non-degenerate annulus around $\infty$. Let $z$ be a (pre-)pole at level $l$, i.e.\ $l$ is the minimal number so that $N^{\circ l}(z) = \infty$. Since the fiber at $\infty$ is trivial, we can choose $m = m(l) \ge 0$ so that the map $N^{\circ l} \colon \inter P_{m+l}(z) \sm \{z\} \to \inter P_m(\infty) \sm \{\infty\}$ is a covering, say of degree $b \ge 1$. As $\fib(\infty) = \{\infty\}$, there exists a sequence $\left(A_{n_i}(\infty)\right)$ of nested annuli with $\sum_i \modu A_{n_i}(\infty) = \infty$ and $n_i > m$. Using the covering $N^{\circ l} \colon \inter P_{m+l}(z) \sm \{z\} \to \inter P_m(\infty) \sm \{\infty\}$, we can pull this sequence back to construct the sequence $(B_i(z))$ of nested annuli around $z$. By property (\ref{It:PuzPr3}), $B_i(z) = \inter P_{n_i + l}(z) \sm P_{n_i + k_i + l}(z)$, where $n_i$ and $k_i$ are the indexes used in the definition of $A_{n_i}(\infty)$. By the covering property of moduli, $\modu B_i(z) = \modu A_{n_i}(\infty) / b$, and thus the sequence $(B_i)$ will have the divergent sum of moduli as well. Therefore, $\fib(z) = \{z\}$ by the Gr\"otzsch inequality.

Finally, let us prove property (\ref{It:PuzPr5}). Assume the contrary, and let $z \in \fib(x)$ be a point that belongs to $\partial P_n(x)$ for all sufficiently large $n$. Observe that $z$ must be a pole or a prepole (we can exclude $\infty$). But we know by property (\ref{It:PuzPr4}) that $\fib(z) = \{z\}$. Therefore, there must be a large index $k$ such that $P_k(z)$ and $P_k(x)$ are disjoint, a contradiction.
\end{proof}

\subsection{Combinatorially recurrent points and renormalization of periodic orbits}

Let us fix $g := N_p^{\circ M}$ to be an iterate of the Newton map $N_p$ for which we have well-defined puzzles of any depth.

We say that $x$ is a \textit{combinatorially recurrent point} if $x$ is not eventually mapped to $\infty$ or attracted to the roots of $p$, and for every puzzle piece $P_n(x)$ of depth $n$ there exists $m \geqslant 1$ ($m$ depends on $n$) such that $g^{\circ m}(x) \in P_n(x)$. By pulling back along the orbit of $x$ we can define an infinite sequence $P_{n_0}(x) \supset P_{n_1}(x) \supset P_{n_2}(x) \supset \ldots $, $k_i := n_{i+1} - n_i > 0$, of puzzle pieces (which we will call a \textit{nest}) such that $k_i$ is the first return time from $P_{n_{i+1}}(x)$ to $P_{n_i}(x) = g^{k_i}(P_{n_{i+1}}(x))$ (here $n_0$ is arbitrary, while the increasing sequence $\left(n_i\right)_{i > 0}$ depends on $n_0$ and $x$).

\begin{proposition}[Combinatorially recurrent points are well inside]
\label{Prop:PuzRec}
For every $n_0 \geqslant 0$, if $x$ is combinatorially recurrent, and $\left(P_{n_i}(x)\right)_{i \geqslant 0}$ is the nest built by pulling back $P_{n_0}(x)$ along the orbit of $x$, then there exist $n_s$ and $n_j$ large enough such that $P_{n_{s}}(x) \subset \inter P_{n_j}(x)$. 
\end{proposition}

\begin{proof}
To prove this proposition, first observe that if the boundaries of two puzzle pieces intersect, then they intersect along at least one common (pre-)pole or $\infty$. Moreover, the only periodic (fixed) point of $g$ that may lie on the boundary of a puzzle piece is $\infty$. Since $\infty$ has a trivial fiber (property (\ref{It:PuzPr4}) of Theorem~\ref{Thm:PuzPr}), we can assume that for some $j \geqslant 0$ the boundary $\partial P_{n_j}(x)$ is disjoint from $\infty$, and thus contains no periodic points of $g$.

We claim that there exists $n_s$ with $s > j$ so that $\partial P_{n_s}(x) \cap \partial P_{n_j}(x) = \emptyset$. Assume the contrary. Since the puzzle pieces $P_{n_i}(x)$ are nested, our assumption implies that there exists a (pre-)pole $z \in \partial P_{n_j}(x)$ such that $z \in \partial P_{n_i}(x)$ for all $i \geqslant j$. But since every puzzle piece $P_{n_i}(x)$ for $i \geqslant j$ is mapped by some iterate $g^{\circ r_i}$ onto $P_{n_j}(x)$ (with $r_0 := 0$), the boundary of $P_{n_j}(x)$ must contain a (pre-)pole of the form $g^{\circ r_i}(z)$ for all $i \geqslant j$. This is clearly impossible because $r_i$ grows as $i$ grows and there are no periodic points on $\partial P_{n_j}(x)$. This contradiction implies the claim, and finishes the proof of the proposition.  
\end{proof}

A periodic point of period at least $2$ serves as an obvious example of a combinatorially recurrent point. Now we want to show how to extract a polynomial-like map for some periodic orbits. Let us briefly review the standard theory of polynomial-like maps. These were introduced by Douady and Hubbard \cite{PolyLike} and have played an important role in complex dynamics ever since. 

\begin{definition}[Polynomial-like maps]
\label{Def_PolyLikeMaps} A \emph{polynomial-like} map of degree $d \geqslant 2$ is a triple
$(f,U,V)$ where $U,\,V$ are open topological disks in $\Cc$, 
the set $\overline{U}$ is a compact subset of $V$, and $f: U \to
V$ is a proper holomorphic map such that every point in $V$ has $d$
preimages in $U$ when counted with multiplicities.
\end{definition}

\begin{definition}[Filled Julia set] 
\label{Def:FilledJulia}
Let $f: U \to V$ be a polynomial-like map. The \emph{filled
Julia set} of $f$ is the set of points in $U$ that never
leave $V$ under iteration of $f$, i.e.
\begin{equation*}
K(f) = K(f,U,V) = \bigcap_{n=1}^\infty f^{-n}(V).
\end{equation*}
As with polynomials, we define the \emph{Julia set} as $J(f)=\partial{K(f)}$.
\end{definition}

The simplest example of a polynomial-like map comes from restricting an actual polynomial: let $p$ be a polynomial of degree $d \geqslant 2$, let $V=\{z\in \mathbb{C}: |z|<R \}$ for sufficiently large $R$ and $U=f^{-1}(V)$. Then $p \colon U \to V$ is a polynomial-like mapping of degree $d$.

Two polynomial-like maps $f$ and $g$ are \emph{hybrid equivalent} if
there is a quasiconformal conjugacy $\psi$ between $f$ and $g$ that is
defined on a neighborhood of their respective filled Julia sets
so that $\bar{\partial}\psi = 0$ on $K(f)$.

The crucial relation between polynomial-like maps and polynomials is
explained in the following theorem, due to Douady and Hubbard \cite{PolyLike}. 

\begin{theorem}[The straightening theorem]
\label{Thm:DH}
Let $f \colon U \to V$ be a polynomial-like map of degree $d$.
Then $f$ is hybrid equivalent to a polynomial $P$ of degree $d$.
Moreover, if $K(f)$ is connected and $d \geqslant 2$, then $P$ is unique up to affine
conjugation.\qed
\end{theorem}

As an immediate consequence of the first part of the theorem, it follows that $K(f)$ is connected if and only if $K(f)$ contains the critical points of $f$. 

Now we define the notion of renormalization of rational maps. Let $R$ be a rational map of degree $d$.

\begin{definition}[Renormalization]
\label{Def:Renorm}
$R^{\circ n}$ is called \emph{renormalizable} if there exist open disks $U,V \subset \C$ such that $(R^{\circ n},U,V)$ is a polynomial-like map whose critical points are contained in the filled Julia set of $(R^{\circ n},U,V)$.

Such a triple $\rho:= (R^{\circ n},U,V)$ is called a \emph{renormalization}, and $n$ is the \emph{period} of the renormalization $\rho$.
\end{definition}

The filled Julia set of $\rho$ is denoted by $K(\rho)$, and the critical and postcritical sets by $C(\rho)$ and $P(\rho)$ respectively. The \emph{$i$-th small filled Julia set} is given by $K(\rho,i) = R^{\circ i}(K(\rho))$. The \emph{$i$-th small critical set} is $C(\rho,i) = K(\rho,i) \cap C_R$ for $1 \leqslant i \leqslant n$, where $C_R$ is the critical set of $R$.

The following result shows that a renormalization does not depend on domains provided that a small critical set is fixed \cite[Theorem 7.1]{MC}.

\begin{theorem}[Uniqueness of renormalization] 
\label{Thm:RenormUnique} Let $\rho=(R^n, U, V)$ and $\rho'=(R^n, U', V')$ be two renormalizations of the same period. If $C(\rho,i) = C(\rho',i)$ for all $1 \leqslant i \leqslant n$, then the filled Julia sets are the same, i.e. $K(\rho) = K(\rho')$. \qed
\end{theorem}

Finally, we are in the position to prove the key proposition about the renormalization of periodic points whose fiber contains a critical point.

Recall that the only possible fixed points of a Newton map $N_p$ are $\infty$ or roots of the polynomial $p$. The local dynamics at these points is well-understood, and the following proposition describes the dynamics at higher period points in terms of renormalizations. 

\begin{remark}
Note that repelling periodic postcritical points are also considered in the statement below, and though they are not needed for the Fatou--Shishikura injection (Theorem~\ref{Thm:FSI}), which concerns only with non-repelling cycles, we include them for the later application (see Proposition \ref{Prop:NonRepellingRenormalizableMinimalPeriod}).
\end{remark}

\begin{proposition}[Renormalization at periodic points]
\label{Prop:NonRepellingRenormalizable}
For each attracting-critically-finite Newton map $N_p$ and for every finite subset $Q$ of the set of periodic points of $N_p$, there are a large enough index $n$ and an iterate $M > 0$ so that
\begin{enumerate}
\item if $q \in Q$ is a non-fixed periodic point and $\fib(q)$ contains at least one critical point, the map $N_p^{\circ k M}\colon \inter P_{n + k}(q) \to \inter P_{n}(q)$ is a polynomial-like map of degree $d\ge 2$ with connected filled Julia set for some $k>0$, and
\item for any two non-fixed periodic points $q$ and $q'$ in $Q$, either $\fib(q)=\fib(q')$ or $P_n(q)\cap P_n(q')=\emptyset$.
\end{enumerate}
\end{proposition}

\begin{proof}
Suppose $q$ is an $r$-periodic point of $N_p$ with $r\geqslant 2$, and $\fib(q)$ contains a critical point. Then $q \not\in \Deltat_n$ for all possible $n$, and hence for every such $q$, by property (\ref{It:PuzPr5}) of Theorem~\ref{Thm:PuzPr}, the fiber $\fib(q)$ is a closed connected set disjoint from the boundary of any puzzle piece $P_n(q)$. 

Choose $n$ sufficiently large so that for any two non-fixed periodic points $q, q' \in Q$, either $\fib(q)=\fib(q')$ or $P_n(q)\cap P_n(q')=\emptyset$ (where in the former case it is evident that $P_n(q)=P_n(q')$), and so that $\partial P_n(q)$ is disjoint from $\infty $ for all $q \in Q$. The depth $n$ exists because $\infty$ and all (pre-)poles have trivial fibers (property (\ref{It:PuzPr4}) of Theorem~\ref{Thm:PuzPr}) and $Q$ is finite.  

Let $N$ be as in Theorem~\ref{Thm:PolesInGraph} and $\nu$ be as defined in the paragraph above Lemma~\ref{Lem:PropertiesDeltap}. For every integer $l \geqslant 1$, by properties (\ref{It:PuzPr2}) and (\ref{It:PuzPr3}) of Theorem~\ref{Thm:PuzPr} the map $N_p^{\circ (r l M)}$ with $M = N + \nu - 1$ sends $P_{n+r l}(q)$ onto $P_n\left(N_p^{\circ (r l)}(q)\right) = P_n(q)$, the puzzle pieces are nested as $P_{n+l r}(q) \subset P_n(q)$, and by property (\ref{It:PuzPr1}) of Theorem~\ref{Thm:PuzPr} all such puzzle pieces are closed topological disks.

Suppose $q \in Q$ and $\fib(q)$ contains at least one critical point.  Obviously, $q$ is a combinatorially recurrent point, and therefore by Proposition~\ref{Prop:PuzRec} there exists $l$ large enough so that $P_{n+lr}(q)$ is contained in $\inter P_n(q)$. For such $l$, the mapping $N_p^{\circ{kM}} \colon  \inter P_{n + k}(q) \to \inter P_{n}(q)$ with $k := r l$ is a polynomial-like map with filled Julia set equal to $\fib(q)$, where the existence of some critical point in $\fib(q)$ guarantees that the degree of $N_p^{\circ kM}$ is at least two.
\end{proof}

\begin{remark}
It is immediate from Proposition~\ref{Prop:NonRepellingRenormalizable} that every \emph{non-repelling} periodic point of period at least two is contained in a polynomial-like restriction specified by the proposition, and every two such non-repelling periodic points either have the same fibers, or can be separated by a truncated puzzle of some deep level.
\end{remark}

\subsection{Beyond the attracting-critically-finite case}
\label{SSec:BeyondACF}

The work we did in the previous sections results in Theorem~\ref{Thm:PuzPr}, which is a stronger and more precise version of Theorem~\ref{Thm:B} for the case when the Newton map $N_p$ is attracting-critically-finite, so Definition~\ref{Def_ConcreteChannelDiagram} of channel diagram applies.  In order to establish Theorem~\ref{Thm:B} for general Newton maps, observe that it requires the construction of puzzle pieces only in a neighborhood of the Julia set, and for this no essential changes are necessary.

To make this clear, we need to describe the procedure that turns the Newton map $N_p$ of an arbitrary polynomial $p$ into an attracting-critically-finite Newton map $N_{\tilde p}$ (see Proposition~\ref{Prop:ACF}) in a bit more detail. This is a routine surgery that replaces a compact and forward invariant disk neighborhood of a root within its immediate basin by the disk $D_{1/2}(0)$, endowed with the dynamics $z\mapsto z^k$, where $k\ge 2$ is the degree of the self-map within the immediate basin. A similar replacement might be necessary within preperiodic basin components that contain critical points. Except for these finitely many compact and forward invariant disks within the basins, the dynamics of $N_p$ and $N_{\tilde p}$ are topologically conjugate, in particular in a neighborhood of the Julia sets. By an appropriate choice of the equipotential $t\in(0,1)$ in the truncation procedure in Definition~\ref{Def:TruncPuzzle} (i.e., by choosing $t$ sufficiently close to $1$), we can make sure that all puzzle pieces are disjoint from the disks in which the surgery takes place. Therefore, the construction of Newton puzzles for $N_{\tilde p}$ immediately carries over to $N_p$, so Theorem~\ref{Thm:PuzPr} generalizes to all Newton maps. In the process, the required graph $\Gamma$ that is invariant under an iterate $g$ of $N_p$ is obtained from the graph for $N_{\tilde p}$ (as constructed in Section~\ref{SSec:Puz}), truncated by deleting the parts within the finitely many surgery disks, and taking its image under the conjugating homeomorphism (which is defined only on the complement of the surgery disks). The truncated graphs are called $\Deltat_m$ in Section~\ref{SSec:Puz}.

\begin{remark}
The fact that the puzzles and the graphs are, in general, not defined in a neighborhood of the roots has an analogy to polynomial dynamics: the roots have invariant neighborhoods with simple dynamics, like the point at $\infty$ for polynomials, and not constructing the puzzles near the roots is like not constructing polynomial puzzles near $\infty$. In a sense, polynomial puzzles are naturally constructed in the setting of polynomial-like maps; in our case the complement of the surgery disks is a very analogous setting. (If desired, one can construct an analog to the channel diagram for arbitrary Newton maps, but giving up within the surgery disks either invariance or finiteness; this requires a certain effort but provides little insight to the dynamics near the Julia set.)
\end{remark}


\section{Proof of the Fatou--Shishikura injection (Theorem~\ref{Thm:FSI})}
\label{Sec:ProofFSI}

In this section we provide the proof for the Fatou--Shishikura injection for Newton maps. We start by proving the Fatou--Shishikura injection for polynomials. This result is not new, and neither is the idea of the proof we give; we provide it in order to describe the ideas that we later use for Newton maps in a simpler context, and so that it fits nicely with our proofs for Newton maps. An added benefit is that our paper becomes more self-contained. The original reference is of course \cite{Kiwi}.

\begin{proposition}[The Fatou--Shishikura injection for polynomials]
The Fatou--Shi\-shi\-ku\-ra injection holds for every iterated polynomial: for every polynomial there is an injection from the set of its non-repelling periodic orbits to the set of its critical orbits that assigns to each non-repelling periodic orbit a critical orbit that is exclusively dynamically associated to it.
\label{Prop:FatouShishikuraPolynomials}
\end{proposition}

\begin{proof}
Consider a polynomial $p$ of degree $d\ge 2$. Observe that all attracting and parabolic periodic orbits have their own critical orbits that converge to them, so we do not have to treat them in the sequel.

The first principal step is the main result from Goldberg and Milnor \cite[Theorem~3.3]{GM}; it can be phrased as saying that any two non-repelling and non-parabolic fixed points are separated by two fixed rays that land at the same repelling or parabolic periodic point. 

Since the number of non-repelling periodic points is always finite, there is a finite forward invariant set $\mathcal R$ of preperiodic or periodic dynamic ray pairs (two rays together with their common landing point) that separate any two non-repelling periodic points (they all land at repelling or parabolic periodic points, so they avoid the non-repelling periodic points). For a given non-repelling and non-parabolic periodic point $w$, say of period $k$, let $U_{\mathcal R}(w)$ be the neighborhood of points in $\C$ that are not separated from $w$ by any ray pair in $\mathcal R$ 
(equivalently, it is the component of $\C \sm \mathcal R$ containing $w$). This is a simply connected domain in $\C$. 

There is a local branch of the inverse of $p^{\circ k}$ that fixes $w$. If it can be extended to all of $U_{\mathcal R}(w)$, then this induces a holomorphic self-map from $U_{\mathcal R}(w)$ to itself with a fixed point at $w$, and it cannot be surjective (the distance of external angles of the boundary rays becomes smaller by a factor of $d$), so by the Schwarz Lemma $w$ must be an attracting fixed point of the inverse of $p^{\circ k}$. Hence $w$ must be a repelling periodic point of $p$, in contradiction to our hypothesis. Therefore, there must be at least one critical point of $p^{\circ k}$ in $U_{\mathcal R}(w)$, say $c$, with $p^{\circ k}(c)\in U_{\mathcal R}(w)$. There might possibly be more such critical points; call them $c_1,\dots,c_m$.

We claim that at least one of these will have the property that $p^{\circ jk}(c_i)\in U_{\mathcal R}(w)$ for all $j\ge 1$. Indeed, if $p^{\circ jk}(c_i)\not\in U_{\mathcal R}(w)$ for some $i$ and $j$, then $p^{-(jk)}(\mathcal R)$ will contain a ray pair that separates $w$ from $c_i$. So if the claim is false, there is a finite preimage of $\mathcal R$ that separates $w$ from all $c_i$, and for this enlarged set of rays we get the same contradiction as above. So after reducing $m$ if possible, we retain a non-empty set of critical points $\{c_1,\dots,c_m\}$ that remain in $U_{\mathcal R}(w)$ under all iterates of $p^{\circ k}$.

The next step is to claim that, if $w$ is a Cremer point, then at least one of the $c_i$ contains $w$ in its $\omega$-limit set $\omega(c_i)$; and if $w$ is the center of a Siegel disk, then $\bigcup_i\omega(c_i)$ contains the boundary of this Siegel disk.

If this claim is false for a Cremer point, then let $V$ be a simply connected neighborhood of $w$ in the complement of all $\omega(c_i)$. Then for every $j$ there is a branch of the inverse of $p^{\circ j k}$ defined on $V$ that fixes $w$. By the Koebe $1/4$-theorem, the set $W:=\bigcap p^{-(jk)}(V)$ contains a neighborhood of $w$. This implies that all $p^{\circ jk}(W)\subset V$, so $w$ is in the Fatou set of $p$, and this is a contradiction.  

If $w$ is the center of a Siegel disk, say $S$, then we always have a branch of the inverse of $p^{\circ(jk)}$ that fixes $w$ and hence $S$. Let $z$ be a point in $\partial S$ and assume it is not contained in $\bigcup_i\omega(c_i)$. Then we set $V$ to be the union of $S$ and a neighborhood of $z$ so that $V$ is simply connected. On this set, the branch of the inverse of $p^{\circ(jk)}$ is defined that fixes $w$, for all $j$, and these maps form a normal family on $V$. Since $p^{\circ k}$ is an irrational rotation on $S$, there is a subsequence of the inverses of the $p^{\circ(jk)}$ that converges to the identity on $V$. There is thus a neighborhood of $z$ on which arbitrarily large iterates of $p^{\circ k}$ stay in $V$. But this is a contradiction because $z$ is in the Julia set, so each of its neighborhoods contains escaping points.

In both cases, $w$ has one critical point in $\{c_1,\dots,c_m\}$ that is dynamically associated to it, and the separating properties of $\mathcal R$ make sure that the same critical point cannot be dynamically associated to a different non-repelling periodic point of $p$, so it is exclusively dynamically associated. 
\end{proof}

\begin{proof}[Proof of Theorem~\ref{Thm:FatouShishikura}]
All fixed points of $N_p$, except the repelling fixed point at $\infty$, are the roots of the polynomial $p$ and thus attracting or superattracting. These have at least one critical point in their immediate basins, and their critical orbits are exclusively dynamically associated to the respective (super-)attracting fixed points. Therefore the existence of the Fatou--Shishikura injection is obvious for the fixed points (if there are several critical points in any given immediate basin, then there is a choice for this injection). We can thus restrict our attention to non-repelling periodic orbits of periods $2$ or higher.

Let $N_{\tilde p}$ be the attracting-critically-finite Newton map obtained from $N_p$ by Proposition~\ref{Prop:ACF}. If the Fatou--Shishikura injection holds for $N_{\tilde p}$, then it also holds for $N_p$: these two maps are conjugate away from some Fatou neighborhood of the roots, for which the injection has been established above. Hence, we can assume that $N_p$ is attracting-critically-finite. 

By Proposition~\ref{Prop:NonRepellingRenormalizable} (see also the subsequent remark), every non-repelling periodic point of $N_p$ of period at least two is contained in the filled Julia set of some polynomial-like restriction of $N_p$ with connected Julia set, and there exists a sufficiently large index $n$ so that any two such small filled Julia sets are either separated by $\Deltat_n$, or they coincide. In the former case, the accumulation sets of critical orbits from disjoint filled Julia sets are disjoint, and hence a critical orbit in a small filled Julia set can be dynamically exclusively associated only to a non-repelling period orbit within the small Julia set. It therefore suffices to prove the claim of the injection for any polynomial-like map with connected Julia set.  

The Fatou--Shishikura injection is true for polynomials by Proposition~\ref{Prop:FatouShishikuraPolynomials}. Moreover, the property of a critical orbit to be exclusively dynamically associated to a non-repelling periodic orbit is not spoiled by the straightening homeomorphism because it provides a topological conjugacy between two maps in some neighborhood of filled Julia sets (see Theorem~\ref{Thm:DH}). Hence, the injection holds for polynomial-like maps as well, and the claim follows. 
\end{proof}

\section{Postcritically finite Newton maps and other future directions}
\label{Sec:PCFandFuture}

The results on combinatorics of puzzles (obtained in Section~\ref{Sec:Renorm}) will be used in the subsequent work \cite{LMS1} to derive combinatorial properties of postcritically \emph{finite} (pcf) Newton maps of a polynomial. It will be shown that every pcf Newton map has an associated forward invariant graph comprised of three parts: the Newton graph, Hubbard trees, and bubble rays. The Newton graph (in the terminology of the present paper) captures the behavior of all critical points that are eventually fixed, and relying on Corollary \ref{Cor:ConnectedInC} is taken at a sufficiently high level so that the ever-important connectedness in $\mathbb{C}$ holds. Hubbard trees will capture the dynamics of all eventually periodic postcritical points which are not fixed relying on Proposition \ref{Prop:NonRepellingRenormalizable} of the present work. Thirdly, so-called \emph{bubble rays} will connect the Newton graph to Hubbard trees. The combinatorial data provided by such a tri-partite graph will be just enough to reconstruct a Newton map using W.\,Thurston's characterization of rational maps. Moreover, a complete classification of pcf Newton maps will be given in the forthcoming manuscript \cite{LMS2}.

To obtain a classification of pcf Newton maps in \cite{LMS2}, Hubbard trees will be extracted from polynomial-like restrictions around non-fixed postcritical points. However, it is essential for \emph{classification} purposes that the polynomial-like restrictions have smallest possible period. Taken by itself, Proposition \ref{Prop:NonRepellingRenormalizable} would only allow us to extract Hubbard trees from an iterate of the desired polynomial-like mapping, and so with slight modification we strengthen the proposition so that renormalizations have lowest period. Even though this result will be mainly used in the postcritically finite setting, we will formulate and prove it in the general (attracting-critically-finite) case (see Definition~\ref{Def:acf}; note that every pcf Newton map is automatically attracting-critically-finite).

Let $q$ be a non-fixed periodic point of $N_p$. The \emph{period of the fiber} $\fib(q)$ is defined to be the minimal integer $i$ so that $N_p^{\circ i}(\fib(q))=\fib(q)$.  Note that in the trivial case $\fib(q) = \{q\}$ this period is equal to the period of the repelling cycle containing $q$. In the case when $\fib(q)$ is nontrivial, it is a consequence of Proposition \ref{Prop:NonRepellingRenormalizable} that $q$ is renormalizable with filled Julia set (defined in Definition~\ref{Def:FilledJulia}) given by $\fib(q)$.

 For a given renormalization $\rho$ with the filled Julia set $\fib(q)$, the period of $\fib(q)$ divides the period of $\rho$ (see Definition~\ref{Def:Renorm} for the definition of the period of a renormalization). A \emph{lowest period renormalization} is a renormalization whose period coincides with the period of the corresponding fiber. 
 
 Let $Q$ denote the set of critical and postcritical points of $N_p$ that have \textit{finite} orbit and are not in the basin of any root of $p$.

\begin{proposition}[Lowest period renormalization]
\label{Prop:NonRepellingRenormalizableMinimalPeriod} Let $N_p$ be an attracting-critically-finite Newton map. If $q\in Q$ is periodic and $\fib(q)$ is not a point, then $q$ is lowest period renormalizable with filled Julia set $\fib(q)$.
\end{proposition}

\begin{proof} There is a sufficiently deep puzzle so that for all $q,q'\in Q$, either $\fib(q)$ and $\fib(q')$ are equal or are in different puzzle pieces. Some level of the Newton graph contains this puzzle as a subset, and so the Newton graph of this level enjoys the same separation property.

Let $q$ be a periodic point of $N_p$ of period $r \geqslant 2$ so that $\fib(q)$ contains a critical point (if not, then $\fib(q) = \{q\}$, and we are done). From Proposition~\ref{Prop:NonRepellingRenormalizable}, there exist a pair of Jordan disks $U, V$ with $\ovl U \subset V$ and an integer $k$ which is a multiple of $r$, such that $q \in U$ and  $N_p^{\circ k M} \colon U \to V$ is a polynomial-like mapping (with non-escaping critical points). We want to extract a polynomial-like map with connected filled Julia set containing $q$ given by a polynomial-like restriction of exactly $N_p^{\circ k'}$, where $k'$ is the period of the Julia set $\fib(q)$. As we mentioned above, $k' | r$, and let $M' := kM/k'$. In order to extract the lowest period renormalization, we will modify $U$ and $V$ as follows. 

Define inductively $V_0 := V$, $V_{i+1}$ to be the component of $N_p^{-k'}(V_i)$ containing $\fib(q)$ (and hence $q$), $i \in \{0, \ldots, M'-1\}$ (we will obtain $V_{M'} = U$). Put 
$$
U' := \bigcap_{i = 0}^{M'} V_i.
$$
By construction, $U'$ is non-empty (it contains $q$) Jordan disk (as an intersection of finitely many Jordan disks; note that we can assume by passing if necessary to a deeper level of puzzles, that $\partial V_j$ for all $j$ are disjoint from the critical set of $N_p$ and contain no poles of $N_p$, see property~(\ref{It:PuzPr4}) of Theorem~\ref{Thm:PuzPr}). Moreover, $U'$ is exactly a subset of $U$ containing all points that do not escape $U$ under $k'$-th iterate of $N_p$. The map $N_p^{\circ k'} \colon U' \to N_p^{\circ k'}(U')$ is a proper map between two Jordan disks $U'$ and $V' := N_p^{\circ k'}(U')$, and almost gives us a desired polynomial-like restriction for $q$. The only thing we need to ensure is that $\ovl{U'} \subset V'$.   

It is clear that $U' \subset V'$. Indeed, if $z \in \partial U'$, then there exists a $j \in \{1, \ldots, M'\}$ such that $z \in \partial V_j$. Therefore, $N_p^{\circ k'}(z) \in \partial V_{j-1}$, and thus $N_p^{\circ k'}(z)$ can not lie in $U'$, since the latter domain is the intersection of all $V_i$.

By the earlier construction of truncated puzzle pieces (see property~(\ref{It:PDP4}) of Lemma~\ref{Lem:PropertiesDeltap} and Definition~\ref{Def:TruncPuzzle}), $\partial U' \cap \partial V'$ may only consist of prepoles (that is, points in the Julia set of $N_p$). With a slight modification, this nontrivial intersection may be eliminated by a ``thickening construction''. For some $\varepsilon >0$, there is a neighborhood $W$ of $\infty$ so that the diameter of $W$ in the spherical metric is less than $\varepsilon$, and so that $W\subset N_p(W)$ (for instance one could choose $W$ using linearizing coordinates for the repelling fixed point $\infty$). Each prepole $z\in U' \cap \partial V'$ satisfies $N_p^{\circ m}(z) = \infty$ for a minimal choice of $m$ that depends on $z$. Denote by $W(z)$ the component of $N_p^{-m}(W)$ that contains $z$. After possibly passing to a smaller choice of $\varepsilon$, the neighborhoods $W(z)$ for each $z\in \partial U' \cap \partial V'$ are pairwise disjoint, and let $\widehat{V}$ denote the union of $V'$ with the neighborhoods $W(z)$ for each $z\in \partial U' \cap \partial V'$. Let $\widehat{U}$ be the component of $N _p^{- k'}(\widehat V)$ containing $U'$. Then $N_p^{\circ k'} \colon \widehat{U}\to\widehat{V}$ is a polynomial-like mapping, and for small enough $\varepsilon$ its filled Julia set is equal to $\fib(q)$ by Theorem \ref{Thm:RenormUnique}.
\end{proof}

The next proposition asserts separability of fibers that intersect $Q$, and as such is a slight extension of the second part of Proposition \ref{Prop:NonRepellingRenormalizable}. This more general result will be essential for the application of Thurston's theorem in \cite{LMS2}. However, it will be convenient to phrase the statement without reference to fibers. 

The filled Julia set of the renormalization at periodic $q\in Q$ with non-trivial fiber (as in Proposition \ref{Prop:NonRepellingRenormalizableMinimalPeriod}) is denoted by $K(q)$. If $q$ has trivial fiber, it is evidently a repelling cycle, and as a matter of convenience we define $K(q)=\{q\}$. Furthermore, if $q\in Q$ is not in a periodic fiber, we define $K(q)$ to be the component of $N_p^{-i}(K(q'))$ where $i$ is minimal so that $N_p^{\circ i}(K(q))=K(q')$ for some periodic $K(q')$.

\begin{proposition}[Separability of filled Julia sets]
\label{Prop:SeparabilityOfFibers} For all $q\in Q$, the set $K(q)$ does not intersect the Newton graph of any level. Furthermore, there is a level of the Newton graph so that for all $q,q'\in Q$, either $K(q)$ and $K(q')$ are in different complementary components of the Newton graph or $K(q)=K(q')$.
\end{proposition}

\begin{proof} 
It is a consequence of the definition of fiber and the forward invariance of the Newton graph that $\fib(q)=K(q)$ for all $q\in Q$.  The first statement is then a consequence of statement (\ref{It:PuzPr5}) of Theorem \ref{Thm:PuzPr} which asserts that the fiber $\fib(q)$ does not intersect puzzles of any level.

There is a sufficiently deep puzzle so that for all $q,q'\in Q$, either $\fib(q)$ and $\fib(q')$ are equal or are in different puzzle pieces. Some level of the Newton graph contains this puzzle as a subset, and so the Newton graph of this level enjoys the same separation property.
\end{proof}

\subsection{Concluding remarks.}

One might wonder in which generality of rational maps the Fatou--Shishikura injection holds true. We believe that Newton maps could lead the way to results for more general maps than Newton maps coming from polynomials. A first natural class of rational maps for which the results should be true are those rational maps that arise as Newton maps of transcendental entire entire functions $f$; this is the case when $f(z) = p(z) \exp (q(z))$ with polynomials $p$ and $q$. For such $f$ the corresponding Newton maps are rational, and in fact their dynamics is very close to the dynamics of Newton maps coming from polynomials, except that $\infty$ is no longer a repelling but a parabolic fixed point. Such Newton maps were studied by Khudoyor Mamayusupov in \cite{KhudoyorThesis, KhudoyorFundamenta, KhudoyorClassification} in what he calls the ``postcritically minimal case''. It is quite plausible that the results and constructions established in our work can be carried over to rational Newton maps of transcendental functions, based on the observation that the topological and combinatorial structure of the immediate basins are the same as those for Newton maps of polynomials, and this structure is the fundamental ingredient in our arguments.

Moreover, there is current work in progress to decompose large classes of rational maps into \emph{Newton-like} components and \emph{Sierpi\'nski-like} maps. Such a decomposition would allow one to extend results on Newton maps to the respective more general rational maps.

\end{document}